\documentclass[11pt]{article}
\usepackage{geometry}
\usepackage{graphicx}
\usepackage{import}

\usepackage{import}
\usepackage{amsmath, amssymb, bbold, mathtools}
\usepackage{footnote}
\usepackage{acronym}
\usepackage{enumerate}
\usepackage{hyperref}
\usepackage[capitalise]{cleveref}
\usepackage{multirow}
\usepackage{subcaption}

\usepackage{graphicx}
\usepackage{natbib}
 \bibpunct[, ]{(}{)}{,}{a}{}{,}

\usepackage{tikz}
\usetikzlibrary{shapes.geometric}
\usetikzlibrary{matrix,positioning}
\usepackage{pgfplots}
\pgfplotsset{compat=1.15}

\usepackage{mathrsfs}
\usetikzlibrary{arrows}
\usepackage{array}

\usepackage[noline,ruled]{algorithm2e}

\usepackage{amsthm}
\theoremstyle{plain}
\newtheorem{theorem}{Theorem}
\newtheorem{lemma}{Lemma}

\newtheorem{corollary}{Corollary}
\newtheorem{proposition}{Proposition}
\theoremstyle{definition}
\newtheorem{definition}{Definition}
\newtheorem{remark}{Remark}

\crefname{theorem}{Theorem}{Theorems}
\crefname{lemma}{Lemma}{Lemmas}
\crefname{proposition}{Proposition}{Propositions}
\crefname{corollary}{Corollary}{Corollaries}
\crefname{definition}{Definition}{Definitions}

\usepackage{algorithm}
\usepackage{algpseudocode}

\newcommand{\R}[0]{\mathbb R}
\newcommand{\Ker}[0]{\mathrm{Ker}\;}

\renewcommand{\dim}[0]{\mathrm{dim}\;}
\newcommand{\Vect}{\mathrm{span}~}
\newcommand{\proj}[2]{\mathrm{proj}_{#1}\left(#2\right)}

\newcommand{\norm}[1]{\left\lVert#1\right\rVert}

\newcommand{\conv}{\mathrm{conv}}
\newcommand{\dataset}[0]{\mathcal D}
\newcommand{\sufficient}{sufficient decision dataset}
\newcommand{\Sufficient}{Sufficient Decision Dataset}
\newcommand{\suff}{sufficient}
\newcommand{\Suff}{Sufficient}

\newcommand{\cQ}{\mathcal Q}

\newcommand{\dualpoly}[2]{#1^\star\left(#2\right)}

\newcommand{\FD}{\mathrm{FD}}

\newcommand{\spc}[1]{\mathrm{dir}\left(#1\right)}
\newcommand{\ctrue}[0]{c_{\text{true}}}

\newcommand{\Eb}{\mathbb{E}}

\newcommand{\Prob}{{\rm{Prob}}}

\newcommand{\cX}{\mathcal{X}}
\newcommand{\cC}{\mathcal{C}}
\newcommand{\cD}{\mathcal{D}}
\newcommand{\cP}{\mathcal{P}}

\renewcommand{\Re}{\mathbb{R}}
\newcommand{\integ}{\mathbb{N}}

\DeclareMathOperator*{\argmin}{arg\,min}

\newcommand{\dir}[1]{\mathrm{dir}\left(#1\right)}

\usepackage{authblk}

\title{What Data Enables Optimal Decisions? \\An Exact Characterization for Linear Optimization
}

\author{
\vspace{0.1in}
Omar Bennouna 
\; Amine Bennouna
\; Saurabh Amin
\; Asuman Ozdaglar\\
MIT\\
\texttt{\{omarben, amineben, amins, asuman\}@mit.edu}
}
\date{}
\begin{document}

\maketitle

\begin{abstract}
We study the fundamental question of how informative a dataset is for solving a given decision-making task. In our setting, the dataset provides partial information about unknown parameters that influence task outcomes. Focusing on linear programs, we characterize when a dataset is sufficient to recover an optimal decision, given an uncertainty set on the cost vector. Our main contribution is a sharp geometric characterization that identifies the directions of the cost vector that matter for optimality, relative to the task constraints and uncertainty set.
We further develop a practical algorithm that, for a given task, constructs a minimal or least-costly sufficient dataset. 
Our results reveal that small, well-chosen datasets can often fully determine optimal decisions---offering a principled foundation for task-aware data selection.
\end{abstract}

\section{Introduction}
\label{sec:intro}
Decision-making problems are often performed under incomplete knowledge of the state of nature---that is, they rely on parameters that must be learned or estimated. In practice, experts draw on a combination of domain knowledge and experience from previously solved tasks. With the recent surge in data availability, data-driven decision-making has become a dominant paradigm: data now plays a central role in complementing contextual knowledge to guide decisions.
This paper seeks to understand the informational value of a given dataset with respect to a specific decision-making task. More precisely, we ask: to what extent does a dataset enable recovery of the optimal decision, given task structure and prior knowledge?

The fundamental question of data informativeness---or its value---has several important implications. One key implication is data collection: when faced with a new decision-making task, which data should be collected to effectively generalize prior knowledge to the new setting? Ideally, one seeks the smallest---or least costly---yet most informative dataset. A second major implication lies in computing. Recent successes of large-scale models (e.g., LLMs) have been driven by large-scale data and advances in computing. However, computing cost remains a significant bottleneck. Identifying the most informative subset of data for a specific task can significantly reduce dataset size and, consequently, training costs. 
% Rather than training on all available data, models are optimized using only task-relevant samples.
Quantifying data value also impacts mechanism design in data markets and considerations around privacy.

A general setting for studying this question is as follows.
Suppose the decision-maker's goal is to select a decision $x \in \cX$ minimizing a loss $L(x,\theta)$ which depends on an unknown parameter---state of nature---$\theta \in \Theta$. A dataset $\cD = \{q_1,\ldots,q_N\}$ consists of points at which the loss is evaluated: it provides observations $\{L(q_1,\theta) + \epsilon(q_1,\theta), \ldots, L(q_N,\theta) + \epsilon(q_N,\theta)\}$ partially informing on the true state of nature $\theta$, where the random variable $\epsilon(\cdot,\cdot)$ models noisy observations. The decision-maker can then use the observations along with their prior knowledge (set restriction $\theta \in \Theta$) to select a decision $x \in \cX$ minimizing $L(x,\theta)$. The central question is: which datasets $\cD$ allow to recover the task-optimal decision, given the prior knowledge encoded in the uncertainty set $\Theta$.
In the rest of the paper, we study this question in the setting of linear programming---with linear loss $L$ and polyhedral decision set $\cX$---where the task structure enables a sharp analysis.

To illustrate this formalism with an example, consider a hiring problem in which a decision-maker is given a list of candidates and their resumes and decides which subset of candidates to interview in order to reveal their value, and ultimately make a hiring decision.
This problem has been studied in various settings \citep{purohitHiringUncertainty2019, epsteinSelectionOrdering2024}, including within the popular Secretary Problem \citep{kleinbergMultiplechoicesecretary2005,arlottoUniformlyBounded2019, brayDoesMultisecretary2019}. Prior work typically assumes a sequential, adaptive model, where interviews and hiring decisions occur in an online fashion. However, in many real-world scenarios---such as hiring PhD students or faculty---the set of candidates to interview must be chosen in advance, with hiring decisions made afterward based on all interview outcomes. This latter \textit{offline} setting is a natural instance of the data informativeness problem.

Formally, hiring from $d$ candidates is a decision-making problem where a decision consists of a binary vector $x \in  \cX \subset \{0,1\}^d$ indicating which candidates to hire. 
The feasible set $\cX$ encodes organizational constraints, such as a maximum number of hires $\sum_{i=1}^d x_i \leq k$, or maximum expertise-based quotas $\sum_{i \in I_j} x_i \leq k_j$ for subsets $I_j \subseteq [d]$, to name a few. 
Each candidate has an unknown value $\theta_i$, with $\theta \in \Theta \subset \Re^d$ modeling prior information on these values. It consists here in (i) candidates' resumes, which can be seen as features $\phi=(\phi_1,\dots,\phi_d) \in \Re^{l \times d}$, and (ii) historical hiring data, which is pairs of resumes and observed value $(\hat{\phi}_1,\hat{\theta}_1),\ldots,(\hat{\phi}_l, \hat{\theta}_n)$. Specifically, $\Theta = \{\theta \in \Re^d_+ \; : \; \exists \alpha \in \Re^l, \exists \epsilon,\epsilon' \in {\cal E},{\cal E}' \; \text{s.t.} \;  \theta = \alpha^\top \phi + \epsilon, \; \hat{\theta} = \alpha^\top \hat{\phi} + \epsilon' \}$ with ${\cal E} \subset \Re^d,{\cal E}' \subset \Re^l$ noise sets. The loss incurred by a decision $x$ under values $\theta$ is $L(x, \theta) = -\theta^\top x$---the negative total value of selected candidates.
A dataset $\cD \subset \{ q \in \{0,1\}^d \; : \; \sum_{i=1}^d q_i =1\}$ is a subset of candidates to interview, and each interview $q \in \cD$, $q_j=1$, reveals a, possibly noisy, evaluation of a given candidate $j$'s value $L(q,\theta) = \theta_j$ which complements the prior information embedded in $\Theta$. The goal in this application is to select the smallest subset of candidates to interview (dataset) to recover the optimal hiring decision: that is, the smallest, informative dataset for the given task.
Unlike the classical secretary or online selection problems, the interview set must be chosen in advance. This makes the hiring problem an offline data selection task---a natural fit for our informativeness framework.

The question of data informativeness is related to several extensively studied topics in economics, statistics, computer science, and operations research literature. Below, we highlight a few of these areas and the angle with which they approached this question.

\paragraph{Active Learning, Bandits and Adaptive Experimental Design.} 
In many data-driven settings, informativeness is approached via adaptive, sequential data collection. Active learning \citep{settlesActiveLearningLiterature2009} seeks to sequentially select data points that improve a classifier by minimizing predictive loss, while bandit algorithms \citep{lattimoreBanditAlgorithms2020} aim to optimize decisions through sequential exploration. Adaptive experimental design \citep{zhaoExperimentalDesign2024a} similarly selects experiments to maximize information gain about unknown parameters, often guided by Bayesian criteria such as posterior variance reduction. These approaches rely on real-time feedback to guide data acquisition and typically analyze asymptotic behavior. However, in practical applications---such as surveys or field trials---queries must often be selected in advance, and outcomes are revealed only afterward. In such settings, adaptivity is infeasible. 

In contrast to these paradigms, we study fixed datasets in a non-adaptive, finite-sample regime, focusing on geometric conditions for optimal decision recovery rather than statistical estimation error. We show that, even without adaptivity, one can precisely characterize which datasets are sufficient to recover task-optimal decisions---offering an offline analogue to adaptive data selection. 

\paragraph{Blackwell's Informativeness Theory.}
One of the earliest and most celebrated frameworks for comparing datasets is Blackwell's theory of informativeness \citep{blackwellEquivalentComparisonsExperiments1953}. In this framework, a dataset is abstracted as an experiment, which generates a signal $s \in S$ drawn from a distribution \(P(s|\theta)\), informing on $\theta \in \Theta$, the unknown state of nature. This is equivalent to our framing above, with the signal being $s = (L(q_1,\theta) + \epsilon(q_1,\theta), \ldots, L(q_N,\theta) + \epsilon(q_N,\theta))$ and the noise terms specifying the conditional distribution. Two datasets (experiments) are compared by whether one enables better decision-making across \emph{all loss functions} and priors. Formally, an experiment \(P\) is more informative than experiment \(Q\) if
$$
\inf_{\delta:S \to \cX} \Eb_{\theta \sim \pi,\; s \sim P(\cdot|\theta)}[L(\delta(s),\theta)] 
\leq 
\inf_{\delta:S \to \cX} \Eb_{\theta \sim \pi,\; s \sim Q(\cdot|\theta)}[L(\delta(s),\theta)], \quad \text{for all loss $L$ and prior $\pi$}
$$
Blackwell's seminal result shows that this criterion is equivalent to several elegant characterizations, notably through the notion of garbling \citep{deoliveiraBlackwellsinformativeness2018}.
% ---intuitively, that one experiment's signal is a stochastic transformation of the other.

Blackwell's informativeness criterion imposes a strict requirement: it compares datasets by whether they enable better decisions across \textit{all possible tasks}.
In contrast, our work fixes the decision task (loss function $L$ and structure $\cX$) and asks which datasets suffice to recover the task-optimal decision. This restriction aligns better with practical applications but also makes the informativeness question more delicate: as \citet{lecamComparisonExperimentsShort1996} observed, such questions may become ``complex or impossible depending on the statistician’s goal''. Whereas Blackwell compares datasets by their universal utility, our work develops a tractable, task-specific notion of informativeness grounded in the structure of the decision task itself. 

%\vspace{-3mm}

\paragraph{Influence Functions and Robust Statistics.}
Influence functions, originating in robust statistics \citep{huberRobustEstimation1992, hampelRobustStatisticsApproach1986}, quantify the local impact of individual data points on estimators and have recently received renewed interest \citep{broderickAutomaticFiniteSampleRobustness2023}.
Similar approaches include DataShapely \citep{jiangOpenDataValUnified2023, jiangOpenDataValUnified2023} and DataModels \citep{ilyasDatamodelsUnderstanding2022, dassDataMILSelecting2025}.
These methods typically analyze how \textit{small perturbations} to a dataset affect the output of a \textit{fixed estimator}. However, a key limitation of this approach is that data value is generally ``non-additive'': the informativeness of an individual data point is not intrinsic, but rather related to the data set as a whole. Our focus is on the joint informativeness of the full dataset---characterizing when a collection of observations, as a whole, suffices to recover the task-optimal decision. Joint informativeness, combinatorial in nature, is a more challenging problem \citep{freundPracticalRobustnessAuditing2023,rubinsteinRobustnessAuditingLinear2024}. 
Furthermore, while influence functions assess sensitivity in \textit{estimation} problems under \textit{fixed} inference procedures, our framework evaluates data informativeness with respect to a \textit{decision task}, at a dataset-level \textit{independently} of any specific inference or optimization procedure. That is, in contrast to estimator-specific influence methods, our work characterizes dataset-level informativeness based solely on the structure of the decision task. 

Data informativeness is a fundamental problem that relates to multiple literature streams---such as Stochastic Probing \citep{weitzmanOptimalSearch1979,singlaPriceInformationCombinatorial2018, gallegoConstructiveProphet2022}, Optimal Experimental Design \citep{chalonerBayesianExperimental1995,singhApproximationAlgorithms2020}
and Algorithms with Predictions \citep{mitzenmacherAlgorithmsPredictions2020}---but a detailed comparison is beyond the scope of this paper.

\paragraph{Contributions.}
This paper addresses the problem of evaluating the informativeness of datasets relative to a specific decision-making task. We study informativeness in the sense of being able to recover the task's optimal solution.
This problem is challenging: it is combinatorial in nature, requiring assessment of the value of different combinations of data points. 
Moreover, informativeness in decision-making is difficult to quantify. One must identify how information in a dataset is relevant to decisions in the feasible set $\cX$, relative to prior information encoded in the uncertainty set $\Theta$.

To be able to derive precise insights, we focus on tasks that can be formulated as linear programs---a broad and expressive class of decision-making problems whose geometric structure enables precise theoretical analysis. 
Our main contributions are as follows:

\begin{itemize}
    \item 
    \textbf{Geometric Characterization of Dataset Sufficiency:}
We prove a necessary and sufficient condition (\cref{thm: relatively open characterization}) under which a dataset is \textit{sufficient} to recover the optimal decision for a linear program under cost uncertainty. This condition is framed geometrically: a dataset is sufficient if it spans the task-relevant directions that govern what can change the optimal solution, given the structure of the feasible set $\cX$ and the uncertainty set $\Theta$.

\item
\textbf{Constructive Characterization via Reachable Optimal Solutions:}
We show that the span of relevant directions for dataset sufficiency can equivalently be expressed as the span of differences between optimal solutions under different cost vectors in the uncertainty set. This characterization (\cref{thm:span delta is dir x}) provides an algorithmically accessible formulation for evaluating and constructing sufficient datasets.

 \item
\textbf{Efficient Data Collection Algorithm:}
Building on these characterizations, we develop an iterative algorithm that constructs a minimal sufficient dataset. When the uncertainty set is polyhedral, the algorithm terminates in a number of steps equal to the size of the minimal sufficient dataset, and each step involves solving a tractable mixed-integer program.
\end{itemize}

% In this setting, our first main contribution is a tractable characterization: given a decision task represented by a linear program, we derive a necessary and sufficient condition for the datasets to enable the recovery of the task-optimal decisions. \OB{We show that a dataset is sufficient to recover optimal decisions if and only if it allows us to determine the cost difference between any pair of possible minimizers of the objective function—that is, if it spans the set of all directions connecting pairs of optimal solutions (see \cref{cor:dualpoly-carac}).} This characterization reveals how informativeness depends jointly on prior knowledge ($\Theta$) and the task structure ($\cX, L$). Following this result, our second contribution is a tractable data selection algorithm that, for a given task and prior knowledge, outputs the smallest---or least costly---dataset sufficient to solve the task optimally.

Our results reveal that small, well-chosen datasets can often fully determine optimal decisions---offering a principled foundation for offline data selection.

\section{Problem Formulation}\label{sec:problem-formulation}

We study decision-making tasks modeled as linear programs (LPs). That is where the loss $L(x,\theta) = \theta^\top x$ is linear, and the decision set $\cX=\{x\in \R^d,\; Ax=b,\; x\geq 0\}$ is a polyhedron, for $A\in \R^{m\times d},\; b\in \R^m$.
The decision-maker's task is then to solve the LP 
\begin{align}
    \min_{x\in \cX}c^\top x,\label{mainproblem}
\end{align}
where $\cX$ is assumed to be bounded.
The unknown parameter---or state of nature---here is the cost vector $c$.
The decision-maker only knows it to be in some given uncertainty set $\cC\subset \R^d$, which captures prior information on $c$ (these are $\theta$ and $\Theta$). Naturally, this is equivalent to having uncertainty in the constraint right-hand side $b \in \Re^m$ by studying the dual problem instead.

To solve the linear program, the decision-maker can complement their knowledge $c \in \cC$ by data on the task.
A dataset $\cD \subset \R^d$ consists of a set of queries to evaluate the objective function. That is a dataset gives access to the observations $c^\top q$ for $q\in \cD$. We focus on the noiseless setting, where each observation $c^\top q$ is exact. This simplification enables a sharp characterization of informativeness. We then show how the core insights naturally extend to noisy observations in \cref{prop:noisy approximation}.
% Intuitively, we are in the regime $|\cD| \ll d$, with scarce but relatively precise observations. %We will show in \ref{} that the insights derived will not change significantly in the presence of noise. 

The fundamental question we seek to address is which datasets are \textit{sufficient} to solve the linear program. We formalize such a property next. 
Here $\cP(\cX)$ denotes subsets of $\cX$.

\begin{definition}[\Sufficient] \label{def:sufficient}
    A set $\cD:=\{q_1,\dots,q_N\}$ is a \sufficient\ for uncertainty set $\cC$ and decision set $\cX$ if there exists a mapping $\hat X:\R^{N} \longrightarrow \cal P(\cX)$ such that
    \begin{align*}
        \forall c \in \cC, \quad \hat X\left(c^\top q_1,\dots,c^\top q_N\right)=\arg\min_{x\in \cX}c^\top x.
    \end{align*}
When there is no ambiguity on $\cC$ and $\cX$, we simply say that $\mathcal D$ is a \sufficient{}.
\end{definition}

\cref{def:sufficient} states that a dataset is \suff\ if there exists a mapping that can recover the optimal solution of the decision-making task using \textit{only} the dataset’s observations and prior information ($c \in \cC$). 
 Importantly, data collection here is \textit{non-adaptive}: the decision-maker has to commit to a set of queries in advance (rather than sequentially) and then decides based on all the revealed outcomes.

Naturally, $\cD = \{e_1,\ldots,e_d\}$, where $(e_i)_{i\in [d]}$ are canonical basis vectors is a \sufficient. In fact, observing $c^\top e_i = c_i$ for all $i \in [d]$ amounts to fully observing $c$, and solving the linear program with complete information with $\hat{X}((c^\top q)_{q \in \cD}) = \hat{X}(c) := \argmin_{x \in \cX} c^\top x$. The question is then whether there exist other, potentially smaller \suff\ datasets. That is, what is the least amount of information required to solve the task? As we will show, whether a dataset is \suff\ depends critically on the uncertainty set $\cC$ and the feasible region $\cX$, since these jointly determine which directions of $c$ affect the optimal decision. From a data collection and computational efficiency perspective, the smallest such data sets $\cD$ are of interest.

If the goal is to solve the linear program \eqref{mainproblem}, 
a natural relaxation of \cref{def:sufficient} is to require only that a dataset permits recovery of \textit{some} optimal solution, rather than the \textit{entire set} of optimal solutions.
We show in the next proposition that, under mild structural assumptions, this property is equivalent to the property of Definition \ref{def:sufficient}. This means that any dataset that recovers one solution also recovers all solutions. The proof of this equivalence is nontrivial and relies on several structural results we develop later in the paper.

\begin{proposition}[One vs All Optimal Solutions] \label{prop:equivalence-argmin}
        Let $\cC$ be an open convex set and $\cD:=\{q_1,\dots,q_N\}$ a dataset. The following are equivalent:
        \begin{enumerate}
            \item \label{item:fullargmin} There exists a mapping $\hat X:\R^{N} \longrightarrow\cal P(\cX)$ such that
            \(
                \forall c \in \cC, \; \hat X\left(c^\top q_1,\dots,c^\top q_N\right)=\arg\min_{x\in \cX}c^\top x.
            \)

            \item \label{item:singlesol} There exists a mapping $\hat{x}:\R^N\longrightarrow \cX$ such that
                \(
                    \forall c\in \cC,\; \hat{x}\left(c^\top q_1,\dots,c^\top q_N\right)\in \arg\min_{x\in \cX}c^\top x. \label{condition for suff}
                \)
        \end{enumerate}
    
    \end{proposition}

% \begin{remark}[Extentions]
%     At this stage, several possible natural extensions of this model appear. For instance, one can consider more general optimization problems beyond linear optimization. One can also consider approximate optimality rather than exact optimality in \cref{def:sufficient}. We discuss these exciting directions in Section \ref{sec:extensions} \OB{extensions section missing}. For now, we will show that within this model, we obtain already very interesting insight for a wide class of decision-making problems: little well chosen data, can significantly generalize decision-making knowledge.
% \end{remark}

 Notice that observing $c^\top q$ for all $q\in \cD$ is equivalent to observing the projection of $c$ onto the span of $\cD$. This implies that \cref{def:sufficient} is equivalent to the following characterization, which gives a valuable perspective. For any subspace $F\subset \R^d$ and $u\in \R^d$, we denote $u_F$ the projection of $u$ in $F$.
\begin{proposition}\label{prop:sufficient:projections}
    $\dataset :=\{q_1,\dots,q_N\}$ is a \sufficient\ for uncertainty set $\cC$ and decision set $\cX$ if and only if 
    \begin{align*}
        \forall c,c'\in \cC, \quad c_{\Vect \cD}=c'_{\Vect \cD}\Longrightarrow \arg\min_{x\in \cX}c^\top x = \arg\min_{x\in \cX}c'^\top x.
    \end{align*}
\end{proposition}

In words, \cref{prop:sufficient:projections} formulate that a dataset $\cD$ is \suff\ if any two cost vectors that are equivalent from the perspective of the information provided by $\cD$ (and $\cC$) lead to the same optimal solutions in the decision-making problem.

This characterization suggests a natural algorithm for solving the LP \eqref{mainproblem} when given a sufficient dataset \(\mathcal{D} = \{q_1, \ldots, q_N\}\).
Suppose we observe values \(o_i = c^\top q_i, i\in [N]\) for an unknown cost vector \(c \in \mathcal{C}\). We then compute $\hat{c} \in \argmin \{ \sum_{i=1}^{N} (c'^\top q_i - o_i)^2 \; : \; c' \in \mathcal{C} \}$ and use \(\hat{c}\) to solve the decision problem \(\min_{x \in \mathcal{X}} \hat{c}^\top x\). This procedure recovers the projection of \(c\) onto \(\Vect \mathcal{D}\) while respecting the prior of \(\mathcal{C}\). This ensures \( \hat{c}_{\Vect \mathcal{D}} = c_{\Vect \mathcal{D}} \) as $c \in \cC$, and since the dataset is sufficient, guarantees that the resulting decision is task-optimal (\cref{prop:sufficient:projections}).

When the observations are noisy, a sufficient dataset can still yield a correct decision.
In particular, estimating an approximate cost vector \(\hat{c}\) via least-squares from noisy observations still leads to an optimal decision, as long as the noise is sufficiently small.

\begin{proposition}[Noisy Observations] \label{prop:noisy approximation}
    Let $\cC\subset \R^d$ be an open set, and $\cD:=\{q_1,\dots,q_r\}$ a \sufficient{} for $\cC$. Let $c\in \cC$. Let $\varepsilon_1,\dots,\varepsilon_r\in \R$, and for all $i\in [r]$, $o_i=c^\top q_i + \varepsilon_i$. Let $\hat{c}\in \argmin \{\sum_{i=1}^{r}(c'^\top q_i - o_i)^2 \; : \; c' \in \cC\}$. There exists $\kappa>0$ such that if $\norm{\varepsilon}<\kappa$, then $\arg\min_{x\in \cX}\hat{c}^\top x \subset \arg\min_{x\in \cX}c^\top x.$
\end{proposition}

%In other words, decisions made using $\hat{c}$ remain optimal for the true cost vector $c$, despite observation noise. 

\section{Characterizing \Suff\ Datasets}

Given an uncertainty set $\cC$ and a decision set $\cX$, we would like to characterize \sufficient{}s and eventually construct such datasets. As in Blackwell's theory, the difficulty of such characterizations depends on the richness of the uncertainty set $\cC$. In fact, the first results by \cite{blackwellComparisonreconnaissances1949, blackwellComparisonExperiments1951} and 
\cite{shermanTheoremHardy1951} were for a set $\cC$ with only two elements. That is, the data needs to distinguish only two alternative states of nature. The result was later extended to the finite sets by \cite{blackwellEquivalentComparisonsExperiments1953} and then to an infinite sets with regularity conditions by \cite{boll_comparison_1955}.

\subsection{Characterization Under No Prior Knowledge}

We begin with the case of no prior knowledge, i.e., \(\mathcal{C} = \Re^d\), which isolates how the structure of the decision set \(\mathcal{X}\) alone determines what information is necessary to recover the optimal solution. We will then study the case of convex sets. To formulate our result, define $F_0=\Vect\{e_i,\;i\in [d],\; \exists x\in \cX,\; x_i\neq 0\}$ where $e_i$ is the $i-$th element of the canonical basis. $F_0$ captures the coordinates that can take non-zero values in feasible solutions of $\cX$. That is $F_0^\perp$ captures coordinates that are identically zero in all feasible solutions: $e_i \in F_0^\perp \implies \forall x \in \cX, \; x_i=0$.

\begin{proposition} 
 \label{prop:suff:vectorspace}
 Suppose $\cC = \Re^d$.
 $\dataset$ is a \sufficient\ if and only if $F_0 \cap \Ker A \subset \Vect \cD$.
Furthermore, when the condition $F_0 \cap \Ker A \subset \Vect \cD$ is not satisfied, for any mapping $\hat{x}:\R^N\longrightarrow \cX$, and any $K>0$, there exists $c\in \R^d$ such that $c^\top \hat{x}\left(c^\top q_1,\dots,c^\top q_N\right)\geq K + \min_{x\in \cX}c^\top x$. 
\end{proposition}

\cref{prop:suff:vectorspace} indicates that, already with no prior knowledge, not all the information on $c$ is required to solve the optimization problem. In fact, the dataset needs only to capture ``relevant'' information for the decision-making task, defined by $\cX$. The proposition shows that these are the directions in the null space of $A$ ($\Ker A$), that act on active variables ($F_0$). 
% An intuitive illustration of the result is as follows:
% \[
% \underbrace{F_0 \cap \Ker A}_{\text{what we need to know}} \subset \underbrace{\Vect \cD}_{\text{what we measure}}
% \]

Let us provide an intuitive explanation for this result. Since every \(x \in \cX\) satisfies \(x_i = 0\) for all \(e_i \in F_0^\perp\), the components of \(c\) along \(F_0^\perp\) do not influence the objective and the dataset \(\cD\) need not capture these directions. Hence, we can, without loss of generality, restrict attention to the subspace \(F_0\) and replace the variable \(x\) with its projection \(x_{F_0}\). The objective function can be decomposed as $x\longmapsto c_{\Ker A}^\top x + c_{(\Ker A)^\perp}^\top x$. 
Because $\cX$ lies in an affine space parallel to \(\Ker A\), any change from a feasible decision $x$ to another feasible decision $x+\delta$ necessarily verifies $\delta \in \Ker A$.
Therefore, \(c_{(\Ker A)^\perp}^\top x\) is constant across all feasible solutions. This means that only the projection of $c$ in $\Ker A$ matters when comparing costs of feasible decisions.

% For the dataset to be \suff, it needs to be able to compare the costs of each two possible optimal solutions $x_1,x_2 \in \cX$. In fact, any extreme point $x \in \cX$ could be an optimal solution given $c$ is unconstrained $\cC = \Re^d$. Now the difference of extreme points $\delta = x_1 - x_2$ are the extreme directions of the polyhedron, that is directions that move from one solution to another while maintaining feasibility. If the dataset does not capture this direction, ie $\delta \notin \Vect \cD$, then it can not measure the change of cost along this direction as $c^\top x_1 - c^\top x_2 = c^\top \delta$. $\cD$ in fact provides partial observation of $c$ in the form of $c^\top q$ for any $q \in \Vect \cD$.
% From this reasoning, it appears intuitively that the dataset must capture all extreme directions of the polyhedron, hence also their span. The span of these directions in exactly $\Ker A$ which is the affine hull of the polyhedron. This is the space of directions between feasible solutions in the polyhedron.

The second part of \cref{prop:suff:vectorspace} formalizes a dichotomy: either the dataset is sufficient and enables optimal decision recovery, or any algorithm may incur arbitrarily large suboptimality in the worst case. This sharp divide, however, is specific to the unstructured case \(\mathcal{C} = \mathbb{R}^d\). As we will see next, imposing structure on \(\mathcal{C}\) significantly enriches the notion of sufficiency.  %\OB{Intuition:if the dataset $\cD$ misses any essential component of $c$, observations will be unvariant to that component, which can be taken as arbitrarily large without changing observations and hence yielding arbitrarily bad decisions.}

\subsection{Characterization under Convex Uncertainty Sets}

The goal now is to characterize \suff\ datasets for any convex uncertainty set $\cC$. We start by introducing some geometric notions that are useful to understand the sufficiency of a dataset.

\begin{definition}[Extreme Points]
An element $x\in \cX$ is an extreme point if and only if there are no $\lambda \in (0,1)$ and $y,z\in \cX$ such that $x=\lambda y + (1-\lambda)z$. The set of extreme points of $\cX$ is denoted $\cX^\angle$.
\end{definition}

From every extreme point, there is a set of \textit{feasible directions} that allow changing the solution while remaining in the polyhedron $\cX$---the feasible region. Out of these feasible directions, \textit{extreme directions} allow moving to ``neighboring'' extreme points.

\begin{proposition}[Feasible and Extreme Directions] \label{feasible directions polyhedral cone}
     For every $x^\star\in \cX^\angle$, we denote $$\FD(x^\star)=\{\delta \in \R^d,\; \exists \varepsilon>0,\;x^\star +\varepsilon \delta \in \cX\}$$
     the set of feasible directions from $x^\star$ in $\cX$. $\FD(x^\star)$ is a polyhedral cone and $\FD(x^\star)\subset F_0 \cap \Ker A$. We denote $D(x^\star)$ the set of extreme directions of $\FD(x^\star)$: non-zero vectors in $\FD(x^\star)$ that cannot be written as a convex combination of two non-proportional elements of $\FD(x^\star)$.
     
\end{proposition}
In linear programs, optimal solutions are attained in extreme points $\cX^\angle$.
Every extreme point is associated with a set of cost vectors $c$ for which it is optimal. This set forms a cone, as illustrated in \cref{fig: optimality cones} (middle).

\begin{proposition}[Optimality Cones]\label{optimality cone def}
    For every $x^\star\in \cX^\angle$, we denote $\Lambda(x^\star)=\{c\in \R^d \; :\; x^\star \in \arg\min_{x\in \cX}c^\top x\}$. We have 
    \(
        \Lambda(x^\star)=\{c\in \R^d,\; \forall \delta \in D(x^\star),\ c^\top \delta\geq 0\}.
    \)
    For every $\delta \in D(x^\star),$ we denote $F(x^\star,\delta):=\Lambda(x^\star)\cap \{\delta\}^\perp$ the face of the cone $\Lambda(x^\star)$ that is perpendicular to $\delta$. Furthermore, $\Lambda(x^\star)$ is the dual cone of $\FD(x^\star)$.
\end{proposition}

Notice that since $\cX$ is bounded, for any $c\in \R^d$, there exists $x^\star \in \cX^\angle$ such that $c\in \Lambda (x^\star)$, and consequently $\R^d=\bigcup_{x^\star \in \cX^\angle}\Lambda(x^\star)$ as illustrated in \cref{fig: optimality cones} (middle). Neighboring cones share boundaries corresponding to their faces (\cref{fig: optimality cones}, right), where multiple solutions can be optimal.

\begin{figure}[h] 
    \centering
    \begin{minipage}{0.33\textwidth} 
        \centering
        \scalebox{0.6}{              \begin{tikzpicture}[scale=1.5]

% Define the vertices with non-orthogonal angles
\coordinate (A) at (0,0);
\coordinate (A1) at (1,-2/2);
\coordinate (A2) at (1,1/2);
\coordinate (B) at (2,2);
\coordinate (B1) at (2 -1*0.4,2-2*0.4);
\coordinate (B2) at (2 +1*0.7,2 -1*0.7);
\coordinate (C) at (4,1);
\coordinate (C1) at (4 -1*0.5,1-2*0.5);
\coordinate (C2) at (4 -2.5*0.4,1 +1*0.4);
\coordinate (D) at (3,-1.5);
\coordinate (D1) at (3 -2.5*0.4,-1.5 +1*0.4);
\coordinate (D2) at (3 -1/2 * 0.5,-1.5 +2 *0.5);
\coordinate (E) at (1,-2);
\coordinate (E1) at (1 +2 * 0.5,-2+1 * 0.5);
\coordinate (E2) at (1-1/2 * 0.5,-2+2 *0.5);

\node[anchor=east, text=red, xshift=-5pt] at (A) {\Large $x_1$};
\node[anchor=south, text=blue, yshift=5pt] at (B) {\Large $x_2$};
\node[anchor=west, text=green,xshift=5pt] at (C) {\Large $x_3$};
\node[anchor=north west, text=orange] at (D) {\Large $x_4$};
\node[anchor=north, text=purple, yshift=-5pt] at (E) {\Large $x_5$};

% Draw the polyhedron (polygon in 2D) with non-orthogonal angles
\draw[thick, fill=gray!20] (A) -- (B) -- (C) -- (D) -- (E) -- cycle;

\draw[ thick, blue, ->] (1,-2) -- (1 - 0.447,-2 + 0.894);
\draw[ thick, blue, ->] ( 4,1 ) -- ( 4 -0.894 ,1 +0.447 );

\draw[ thick,blue, ->] ( 1,-2)-- ( 1 +0.970,-2 + 0.244);

\node[anchor=center,blue, font=\large] at (3.6,1.4) {$\delta$};
\node[anchor=center,blue, font=\large] at (0.6,-1.7) {$\delta'$};
\node[anchor=center,blue, font=\large] at (1.6,-2.05) {$\delta''$};

% Color each extreme point differently
\fill[red] (A) circle (3pt);
\fill[blue] (B) circle (3pt);
\fill[green] (C) circle (3pt);
\fill[orange] (D) circle (3pt);
\fill[purple] (E) circle (3pt);

% Draw dual cones for each extreme point
% Dual cone at A
\draw[red, thick, dashed] (A) -- (A1);
\draw[red, thick, dashed] (A) -- (A2);
\fill[red, opacity=0.2] (A) -- (A1) -- (A2) -- cycle;

% Dual cone at B
\draw[blue, thick, dashed] (B) -- (B1);
\draw[blue, thick, dashed] (B) -- (B2);
\fill[blue, opacity=0.2] (B) -- (B1) -- (B2) -- cycle;

% Dual cone at C
\draw[green, thick, dashed] (C) -- (C1);
\draw[green, thick, dashed] (C) -- (C2);
\fill[green, opacity=0.2] (C) -- (C1) -- (C2) -- cycle;

% Dual cone at D
\draw[orange, thick, dashed] (D) -- (D1);
\draw[orange, thick, dashed] (D) -- (D2);
\fill[orange, opacity=0.2] (D) -- (D1) -- (D2) -- cycle;

% Dual cone at E
\draw[purple, thick, dashed] (E) -- (E1);
\draw[purple, thick, dashed] (E) -- (E2);
\fill[purple, opacity=0.2] (E) -- (E1) -- (E2) -- cycle;

\node[anchor=center, text=red, font=\large] at (0.7, 0) {$\Lambda(x_1)$};
\node[anchor=center, text=blue, font=\large] at (2.1, 1.4) {$\Lambda(x_2)$};
\node[anchor=center, text=green, font=\large] at (3.5, 0.8) {$\Lambda(x_3)$};
\node[anchor=center, text=orange, font=\large] at (2.5, -1) {$\Lambda(x_4)$};
\node[anchor=center, text=purple, font=\large] at (1.3,-1.5) {$\Lambda(x_5)$};

\end{tikzpicture}}
        \end{minipage}\hfill
    \begin{minipage}{0.33\textwidth} 
        \centering
        \scalebox{0.6}{        \begin{tikzpicture}[scale=1.5]

% Define the vertices with non-orthogonal angles
\coordinate (O) at (0,0);
\coordinate (A1) at (0.707*2.3,-0.707*2.3);
\coordinate (A2) at (0.894*2.3,0.447*2.3);

\coordinate (B1) at ( -0.447*2.3,-0.894*2.3);
\coordinate (B2) at (A1);

\coordinate (C1) at (B1);
\coordinate (C2) at ( -0.875*2.3, 0.25*2.3);

\coordinate (D1) at (C2);
\coordinate (D2) at ( -0.244*2.3,0.970*2.3);

\coordinate (E1) at (A2);
\coordinate (E2) at (D2);

%\node[anchor=east, text=red, xshift=-5pt] at (A) {\Large $x_1$};
%\node[anchor=south, text=blue, yshift=5pt] at (B) {\Large $x_2$};
%\node[anchor=west, text=green,xshift=5pt] at (C) {\Large $x_3$};
%\node[anchor=north west, text=orange] at (D) {\Large $x_4$};
%\node[anchor=north, text=purple, yshift=-5pt] at (E) {\Large $x_5$};

% Color each extreme point differently

% Draw dual cones for each extreme point
% Dual cone at A
\draw[ thick, dashed] (O) -- (A1);
\draw[ thick, dashed] (O) -- (A2);
\fill[red, opacity=0.2] (O) -- (A1) -- (A2) -- cycle;

% Dual cone at B
\draw[ thick, dashed] (O) -- (B1);
\draw[ thick, dashed] (O) -- (B2);
\fill[blue, opacity=0.2] (O) -- (B1) -- (B2) -- cycle;

% Dual cone at C
\draw[thick, dashed] (O) -- (C1);
\draw[ thick, dashed] (O) -- (C2);
\fill[green, opacity=0.2] (O) -- (C1) -- (C2) -- cycle;

% Dual cone at D
\draw[ thick, dashed] (O) -- (D1);
\draw[ thick, dashed] (O) -- (D2);
\fill[orange,opacity=0.2] (O) -- (D1) -- (D2) -- cycle;

% Dual cone at E
\draw[ thick, dashed] (O) -- (E1);
\draw[ thick, dashed] (O) -- (E2);
\fill[purple, opacity=0.2] (O) -- (E1) -- (E2) -- cycle;

\node[anchor=center, text=red, font=\large] at (1.3, 0) {$\Lambda(x_1)$};
\node[anchor=center, text=blue, font=\large] at (0.3, -1) {$\Lambda(x_2)$};
\node[anchor=center, text=green, font=\large] at (-1, -0.5) {$\Lambda(x_3)$};
\node[anchor=center, text=orange, font=\large] at (-1, 1) {$\Lambda(x_4)$};
\node[anchor=center, text=purple, font=\large] at (0.5,1) {$\Lambda(x_5)$};
\end{tikzpicture}}
        \end{minipage}
        \begin{minipage}{0.33\textwidth} 
        \centering
        \scalebox{0.6}{\begin{tikzpicture}[overlay,scale=1.5]

% Define the vertices with non-orthogonal angles
\coordinate (O) at (0,0);
\coordinate (A1) at (0.707*2.3,-0.707*2.3);
\coordinate (A2) at (0.894*2.3,0.447*2.3);

\coordinate (B1) at ( -0.447*2.3,-0.894*2.3);
\coordinate (B2) at (A1);

\coordinate (C1) at (B1);
\coordinate (C2) at ( -0.875*2.3, 0.25*2.3);

\coordinate (D1) at (C2);
\coordinate (D2) at ( -0.244*2.3,0.970*2.3);

\coordinate (E1) at (A2);
\coordinate (E2) at (D2);

%\node[anchor=east, text=red, xshift=-5pt] at (A) {\Large $x_1$};
%\node[anchor=south, text=blue, yshift=5pt] at (B) {\Large $x_2$};
%\node[anchor=west, text=green,xshift=5pt] at (C) {\Large $x_3$};
%\node[anchor=north west, text=orange] at (D) {\Large $x_4$};
%\node[anchor=north, text=purple, yshift=-5pt] at (E) {\Large $x_5$};

% Color each extreme point differently

% Draw dual cones for each extreme point
% Dual cone at A
\draw[ thick, dashed] (O) -- (A1);
\draw[ thick, dashed] (O) -- (A2);
\fill[red, opacity=0.2] (O) -- (A1) -- (A2) -- cycle;

% Dual cone at B
\draw[ thick, dashed] (O) -- (B1);
\draw[ thick, dashed] (O) -- (B2);
\fill[blue, opacity=0.2] (O) -- (B1) -- (B2) -- cycle;

% Dual cone at C
\draw[thick, dashed] (O) -- (C1);
\draw[ thick, dashed] (O) -- (C2);
\fill[green, opacity=0.2] (O) -- (C1) -- (C2) -- cycle;

% Dual cone at D
\draw[ thick, dashed] (O) -- (D1);
\draw[ thick, dashed] (O) -- (D2);
\fill[orange,opacity=0.2] (O) -- (D1) -- (D2) -- cycle;

% Dual cone at E
\draw[ thick, dashed] (O) -- (E1);
\draw[ thick, dashed] (O) -- (E2);
\fill[purple, opacity=0.2] (O) -- (E1) -- (E2) -- cycle;

%\node[anchor=center, text=red, font=\large] at (1.3, 0) {$\Lambda(x_1)$};
%\node[anchor=center, text=blue, font=\large] at (0.3, -1) {$\Lambda(x_2)$};
%\node[anchor=center, text=green, font=\large] at (-1, -0.5) {$\Lambda(x_3)$};
%\node[anchor=center, text=orange, font=\large] at (-1, 1) {$\Lambda(x_4)$};
\node[anchor=center, font=\Huge] at (0,-1) {$\cC$};
\node[anchor=center, font=\Huge] at (0.3,1) {$\cC'$};

\coordinate (O1) at ( -1,-0.25);
\coordinate (O2) at (2,-0.1);
\coordinate (O3) at (0,-3.8);
\coordinate (O4) at (-1,-1.6);

\coordinate (V1) at ( 1.2,0);
\coordinate (V2) at (-2,0.1);
\coordinate (V3) at (-0,3.8);
\coordinate (V4) at (1,1.6);
\draw[thick] 
    (V1) .. controls (V2) and (V3) ..  (V1) ;

%\fill[red] (O1) circle (2pt);
%\fill[red] (O2) circle (2pt);
%\fill[red] (O3) circle (2pt);
%\fill[red] (O4) circle (2pt);
\draw[thick] 
    (O1) .. controls (O2) and (O3) ..  (O1) ;

\draw[ thick, ->] ( -0.447*1.5,-0.894*1.5 ) -- ( -0.447*1.5 -0.894 ,-0.894*1.5 +0.447 );

\draw[ thick, ->] (0.894*1.5,0.447*1.5) -- (0.894*1.5 - 0.447,0.447*1.5 + 0.894);
\draw[ thick, ->] ( -0.244*1.8,0.970*1.8)-- ( -0.244*1.8 +0.970,0.970*1.8 + 0.244);
\node[anchor=center, font=\large] at (-1,-2.3) {$F(x_3,\delta)$};
\node[anchor=center, font=\large] at (-1.1,-1.4) {$\delta$};
\node[anchor=center, font=\large] at (1.3,1.2) {$\delta'$};
\node[anchor=center, font=\large] at (0,2) {$\delta''$};

\node[anchor=center, font=\large] at (-0.5,2.5) {$F(x_5,\delta'')$};
\node[anchor=center, font=\large] at (2.5,1.3) {$F(x_5,\delta')$};

\end{tikzpicture}}
        \end{minipage}\hfill
  \caption{Optimality cones relative to $\cX$ (left), relative to the origin (middle) and examples of the uncertainty sets ($\cC$ and $\cC'$) relative to the optimality cones (right).}
    \label{fig: optimality cones}
\end{figure}
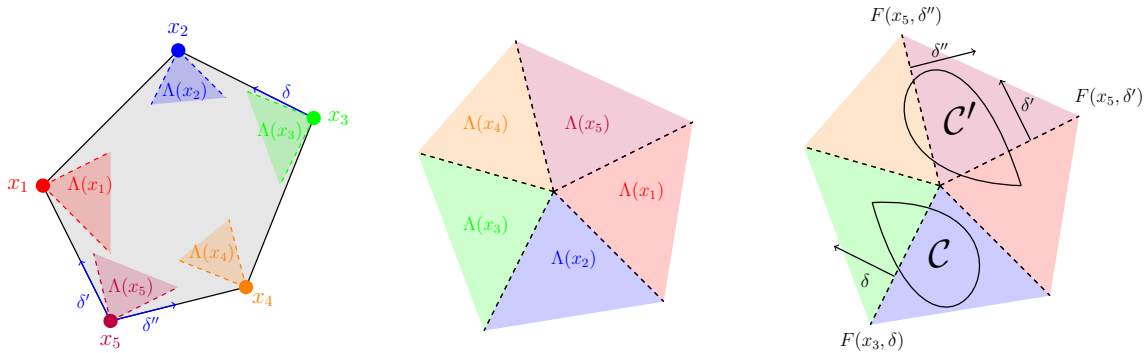

With the notion of optimality cones, solving a linear program for a given cost vector $c$ amounts to finding to which optimality cone it belongs. A dataset is therefore sufficient if it enables to determine the optimality cone of each possible \(c \in \mathcal{C}\). As $\cC$ already restricts the location of its cost vectors, our data only needs to discriminate between cones overlapping with $\cC$ as illustrated in \cref{fig: optimality cones} (right).

%\begin{figure}[h] 
%    \centering
%    \begin{minipage}{0.45\textwidth}
%        \centering
%    \scalebox{\figscale}{\import{figures/}{cones_with_C_1.tex}}
%    \end{minipage}\hfill
%    \begin{minipage}{0.45\textwidth}
%       \centering
%        \scalebox{\figscale}{\import{figures/}{cones_with_C_2.tex}}
%    \end{minipage}
%    \caption{Examples of the cost vector sets ($\cC$ and $\cC'$) relative to the optimality cones (left) and orthogonal directions to optimality cones relative to $\cX$ (right).}
%    \label{fig: optimality cones relative to C}
%\end{figure}

To provide further intuition, consider the example shown in \cref{fig: optimality cones} (right). The set $\cC$ intersects only the cones $\Lambda(x_2)$ (blue) and $\Lambda(x_3)$ (green), hence, the cost vectors can only be in these two cones. Clearly, observing their projection on the span of $\delta$ is sufficient to determine which of the two cones they belong to. The set $\cC'$, however, intersecting $\Lambda(x_4),\Lambda(x_5)$ and $\Lambda(x_1)$, requires projections on the span of both $\delta'$ and $\delta''$. These vectors are not arbitrary; these are extreme directions that move from one cone to its adjacent cone, inducing a face where both cones intersect. The illustration highlights that such vectors are necessary to capture by our data when the face they induce intersects the uncertainty set $\cC$. 
Hence, it is natural to introduce the following set of \textit{relevant extreme directions}.
\begin{definition}[Relevant Extreme Directions]
    Given $\cC\subset \R^d$, we define 
   \begin{align*}
        \Delta(\cX,\cC)=\{\delta \in \R^d \; : \; \exists x^\star\in \cX^\angle,\; \delta \in D(x^\star) \text{ and }F(x^\star,\delta)\cap \cC\neq \varnothing\}.
    \end{align*}
\end{definition}
In the illustration of \cref{fig: optimality cones}, we have $\Vect \Delta(\cX,\cC)=\Vect \{\delta\}$ and $\Vect \Delta(\cX,\cC')=\Vect \{\delta',\delta''\}$, and it is necessary to observe the projections on $\Delta(\cX,\cC)$ and $\Delta(\cX,\cC')$ to recover optimal solutions for uncertainty sets $\cC$ and $\cC'$ respectively. This leads to our first main theorem.
\begin{theorem} \label{thm: relatively open characterization}
            %We assume without loss of generality that $F_0=\R^d$. IMPORTANT: still not sure why I added this assumption 
             Let $\cC$ be an open convex set. $\dataset$ is a \sufficient \;for uncertainty set $\cC$ and decision set $\cX$ if and only if $\Delta(\cX,\cC) \subset \Vect \mathcal D$.
        \end{theorem}
         \begin{proof}
     We denote $F=\Vect \mathcal D$. Notice that we have $\Delta(\cX,\cC) \subset F \Longleftrightarrow \Delta(\cX,\cC) \perp  F^\perp.$ We will now prove that  $\dataset$ is a \sufficient \;for $\cC$ if and only if $ \Delta(\cX,\cC) \perp  F^\perp$.
    \begin{itemize}
        \item $(\Leftarrow)$ Suppose $\Delta(\cX,\cC) \not \perp  \cap F^\perp$. There exists $\delta \in \Delta(\cX,\cC)$ such that $\delta \not\perp  F^\perp$. By definition, there exists $x\in \cX^\angle$ such that $\delta \in D(x)$ and $F(x,\delta) \cap \cC \neq \varnothing$. Let $v\in F(x,\delta) \cap \cC$. Let $\delta_0\in  F^\perp$ such that $\delta_0 ^\top \delta <0$ ($\delta_0$ exists because $\delta \not \perp   F^\perp$). As $\cC$ is open, we can assume without loss of generality that $v+\delta_0\in \cC$ by rescaling $\delta_0$. We know that $v\in F(x,\delta)\subset \Lambda(x)$, and $(v+\delta_0)^\top \delta = \delta_0^\top \delta <0$ which implies that $v+\delta_0\not\in \Lambda(x)$. Finally, since we have $\delta_0 \in F^\perp$, we have $(v+\delta_0)_F=v_F+\delta_{0,F}=v_F$. However, 
        $v \in \Lambda(x)$ and $v+\delta \not \in \Lambda(x)$ which implies $x\in \arg\min_{x'\in \cX}v^\top x'$ and $x\not\in\arg\min_{x'\in \cX} (v+\delta_0)^\top x'$, meaning that $\arg\min_{x'\in \cX} (v+\delta_0)^\top x'\neq \arg\min_{x'\in \cX} v^\top x'$. This
        implies that $\cD$ is not a \sufficient{}
   from \cref{prop:sufficient:projections}.

        \item $(\Rightarrow)$ Suppose $\cD$ is not \suff. From \cref{prop:sufficient:projections}, there exists $c,c'\in \cC$ such that $c_F=c'_F$ and $\arg\min_{x\in \cX}c^\top x \neq \arg\min_{x\in \cX}c'^\top x $. It follows from the definition of the optimality cones $\Lambda$ (\cref{optimality cone def}) that there exists $x\in \cX^\angle$ such that $c\in \Lambda (x)$ and $c'\not\in \Lambda (x)$ (see also \cref{cone equivalence}). For any $\alpha\in[0,1]$, we denote $c_\alpha:=(1-\alpha) c + \alpha c'$
        \begin{align*}
        \alpha^\star :&=\sup\left\{\alpha \in [0,1]\; : \; c_\alpha \in \Lambda(x) \right\} \\
        &= \sup\left\{\alpha \in [0,1]\; : \; c_\alpha^\top\delta \geq 0, \; \forall \delta \in D(x) \right\}.
        \end{align*}
        
        Since $\cC$ is convex, we have $c_{\alpha^\star} \in \cC$.  Since $\Lambda(x)$ is a closed set, we have $c_{\alpha^\star} \in \Lambda(x)$ and hence we have $c_{\alpha^\star} \neq c'$ i.e. $\alpha^\star<1$. Let $\varepsilon\in(0,1-\alpha^\star)$ small enough such that for any $\delta \in D(x)$ such that ${c_{\alpha^\star}}^\top \delta>0$, we have $c_{\alpha^\star+\varepsilon}^\top \delta >0$. As $c_{\alpha^\star + \epsilon} \not \in \Lambda(x)$, there exists $\delta \in D(x)$ for which ${c_{\alpha^\star + \epsilon}}^\top \delta<0$. Such $\delta$ must verify ${c_{\alpha^\star}}^\top \delta = 0$ given the condition defining $\epsilon$. Hence, $c_{\alpha^\star} \in \Lambda(x) \cap \{\delta\}^\perp = F(x,\delta)$, and $c_{\alpha^\star} \in \cC$, which implies $F(x,\delta) \cap \cC \neq \varnothing$ and therefore $\delta \in \Delta(\cX,\cC)$. Moreover, we have $\underbrace{(c_{\alpha^\star+\varepsilon}-c_{\alpha^\star})^\top}_{=\varepsilon(c'-c)\in  F^\perp} \delta =  c_{\alpha^\star+\varepsilon}^\top \delta \neq 0$, i.e. $\delta \not \perp  F^\perp$, and consequently we have $\Delta(\cX,\cC) \not\perp F^\perp$.
        \end{itemize}
        \end{proof}

        \cref{thm: relatively open characterization} is a fundamental characterization of sufficiency, by what information the dataset needs to capture relative to the prior knowledge $\cC$ and the problem structure $\cX$. The result also indicates that such a minimal dataset is, in general, not unique. A careful reader might remark that \cref{thm: relatively open characterization} should imply \cref{prop:suff:vectorspace} when $\cC = \Re^d$. In fact, $\Delta(\cX,\Re^d)$ is the set of all extreme directions of the polyhedron, which indeed precisely spans $\Ker A \cap F_0$.
        
         \begin{remark}
         In Theorem \ref{thm: relatively open characterization}, only convexity is needed for the condition $ \Delta(\cX,\cC) \subset\Vect \mathcal D$ to be sufficient, and only openness is needed for the condition to be necessary. In fact, a more general formulation of \cref{thm: relatively open characterization} would be: (1) If $\cC$ is open then: $\cD$ is \suff\ $\implies$ $\Delta(\cX,\cC) \subset  \Vect \mathcal D$, (2) If $\cC$ is convex then: $\Delta(\cX,\cC) \subset  \Vect \mathcal D$ $\implies$ $\cD$ is \suff.
         \end{remark}

    \subsection{An Algorithmically Tractable Characterization}
    We now develop a second characterization of dataset sufficiency that is particularly well-suited to algorithmic construction.

   The set $\Delta(\cX,\cC)$ of relevant extreme directions of \cref{thm: relatively open characterization} can be seen intuitively as the set of differences $x_1-x_2$ of \textit{neighboring} extreme points $x_1,x_2\in \cX^\angle$, that are optimal for some $c \in \cC$. By relaxing the ``neighboring’’ condition and optimality for a common $c\in \cC$, we arrive at a broader set of directions induced by all pairs of optimal extreme points---which we call \textit{reachable solutions}. This alternative perspective replaces the notion of relevant extreme directions with that of reachable solutions, leading to a representation of sufficiency in terms of differences between optimal solutions under cost vectors in \(\mathcal{C}\).

\begin{definition}[Reachable Solutions]
    Given $\cC\subset \R^d$, we define 
    \begin{align*}
        \dualpoly{\cX}{\cC}:=\left\{x^\star\in \cX^\angle,\;\exists c \in \cC,\; x^\star\in \arg\min_{x\in \cX}c^\top x \right\}=\bigcup_{c\in \cC}^{}\arg\min_{x\in \cX}c^\top x.
    \end{align*} 
and its set of directions as $\dir{\dualpoly{\cX}{\cC}} := \Vect\{x_1 - x_2 \; : \; x_1,x_2 \in  \dualpoly{\cX}{\cC}\}$.
\end{definition}

The set $\dir{\dualpoly{\cX}{\cC}}$ is equal to the span of the set of differences between \textit{any} two elements $x_1,x_2\in \cX$ such that there exists $c_1,c_2 \in \cC$ such that $x_1\in \arg\min_{x\in \cX}c_1^\top x$ and $x_2\in \arg\min_{x\in \cX}c_2^\top x$. By construction, we have $\Vect \Delta(\cX,\cC) \subset \dir{\dualpoly{\cX}{\cC}}$ since each relevant extreme direction corresponds to a direction between optimal solutions. The following theorem shows that these quantities are indeed equal.

\begin{theorem}\label{thm:span delta is dir x}
    For any convex set $\cC\subset \R^d$, we have $\Vect \Delta(\cX,\cC)=\dir{\dualpoly{\cX}{\cC}}$.
\end{theorem}

The converse inclusion proven in this theorem is not immediate. In fact, for a general polyhedron $\cX$ and $\cC$ (see \cref{fig: optimality cones} with $\cC'$ for eg.), $\Delta(\cX,\cC)$ is much smaller than the set of differences of elements in $\dualpoly{\cX}{\cC}$ but their spans are equal. To prove the equality, we prove that for any $x,x'\in \dualpoly{\cX}{\cC}\cap \cX^\angle$, there exists a sequence of extreme points $x_1,\dots,x_h \in \dualpoly{\cX}{\cC}$ such that $x_1=x$ and $x_h=x'$ and for any $i\in [h-1]$, $x_{i+1}-x_i\in \Delta(\cX,\cC)$. In other words, $x_{i+1},x_i$ are neighbors and are both optimal for some $c_i\in \cC$. This implies that every element in \(\dir{\dualpoly{\mathcal{X}}{\mathcal{C}}}\) can be written as a finite linear combination of elements in \(\Delta(\mathcal{X}, \mathcal{C})\), completing the equality. Relating again to \cref{prop:suff:vectorspace} of the case $\cC = \Re^d$, careful linear algebra shows that indeed $\dir{\cX^\star(\Re^d)} = \dir{\cX} = \Ker A \cap F_0$.

\cref{thm: relatively open characterization} implies that to construct a \sufficient\, it suffices to find
a basis of $\dir{\dualpoly{\cX}{\cC}}$ rather than $\Vect \Delta(\cX,\cC)$, which is a much simpler task. The following corollary will indeed be the basis of our algorithm in the next section.

\begin{corollary}\label{cor:dualpoly-carac}
    Let $\cC$ be an open convex set. $\dataset$ is a \sufficient \;for $\cC$ if and only if $\dir{\dualpoly{\cX}{\cC}} \subset \Vect \mathcal D$.
\end{corollary}

\section{A Data Collection Algorithm: Finding Minimal Sufficient Datasets} \label{sec:algorithm}

We now turn to the practical problem of selecting a minimal---i.e., smallest or least costly---dataset $\cD$ that enables generalization from prior contextual knowledge (captured by \(c \in \mathcal{C}\)) to a specific decision-making task (defined by \(\mathcal{X}\)). 
% Corollary~\ref{cor:dualpoly-carac} provides a principled criterion for sufficiency in this setting.

In many practical settings, data collection is subject to constraints on what can be queried. We model this by restricting the dataset to lie in a predefined query set \(\mathcal{Q} \subset \mathbb{R}^d\), so that \(\mathcal{D} \subseteq \mathcal{Q}\). For example, in the hiring problem discussed in \cref{sec:intro}, $\cQ$ is the set of canonical basis vectors---a data point is interviewing one candidate. \cref{cor:dualpoly-carac} implies then that the data collection problem becomes: finding the smallest $\cD \subset \cQ$ verifying $\dir{\dualpoly{\cX}{\cC}} \subset \Vect \cD$.

We will focus in what follows on the important case of $\cQ$ being the set of canonical basis vectors. That is, each query in the data consists in evaluating one coordinate of unknown cost vector $c$, which represents the score of some candidate. In this case, given $\dir{\dualpoly{\cX}{\cC}}$, represented by a basis $v_1,\ldots,v_k$, it is clear that the smallest sufficient data set, verifying the spanning condition of \cref{cor:dualpoly-carac}, is $\cD = \{ e_i \; : \; i \in [d], \; \exists j \in [k], \; v_j^\top e_i \neq 0 \}$. This is all the non-zero coordinates of basis vectors of $\dir{\dualpoly{\cX}{\cC}}$, which are required to span $\dir{\dualpoly{\cX}{\cC}}$. This case can be generalized in a straightforward manner to the case where $\cQ$ is any basis of $\Re^d$; see \cref{alg:data collection algorithm}.

The central step in this approach is therefore to compute $\dir{\dualpoly{\cX}{\cC}}$ and construct a basis for it, which is the focus of the remainder of this section.
We can write 
\(
    \dir{\dualpoly{\cX}{\cC}}=\Vect \{x_0 - x,\; x\in \dualpoly{\cX}{\cC}\}
\)
for some $x_0 \in \dualpoly{\cX}{\cC}$. Hence, to compute $\dir{\dualpoly{\cX}{\cC}}$, we can iteratively add elements of it while ensuring we increase the dimension at every step. This is formalized in Algorithm \ref{alg:meta_alg_dir}.

\begin{algorithm}[h]
\caption{Meta-Algorithm Computing $\dir{\dualpoly{\cX}{\cC}}$}
\label{alg:meta_alg_dir}

    \KwIn{Decision set $\cX$, Uncertainty set $\cC$.}
    \KwOut{A basis $\cD\subset \R^d$ of $\dir{\dualpoly{\cX}{\cC}}$.}
    Initialize $\cD$ to $\varnothing$.
    
    Set $x_0\in \arg\min_{x\in \cX}c_0^\top x$ for some $c_0 \in \cC$.
    
    \textbf{while} there exists $c\in \cC$, $x^\star \in \argmin_{x\in\cX} c^\top x$ such that $x^\star-x_0 \not \in \Vect{\cD}$.

    \quad $\cD \leftarrow \cD \cup \{x^\star-x_0\}$.
    
    \textbf{return} $\cD$
\end{algorithm}

The main step in Algorithm \ref{alg:meta_alg_dir} (condition of the while loop) can be seen as verifying whether the optimization problem
\begin{equation}\label{eq:alg_optimization_pb}
    \sup \{\ \|\proj{(\Vect \cD)^\perp}{x^\star - x_0}\|  \; : \; c \in \cC, \; x^\star \in \argmin_{x \in \cX} c^\top x\},
\end{equation} 
where $\proj{(\Vect \cD)^\perp}{\cdot}$ is the projection map onto $(\Vect \cD)^\perp$, has a solution with a non-zero objective. This optimization problem has two main challenges: first, it entails the inherently difficult task of maximizing a convex objective, and second, it has a bilinear, bi-level constraint $x^\star \in \argmin_{x \in \cX} c^\top x$ as both $c$ and $x^\star$ are variables and $x^\star$ must be an optimal solution to a linear program parameterized by c.

\textbf{Linearizing the objective.}
Remark that if $\alpha$ is a randomly generated Gaussian vector, then any vector $v$, with $\|v\|>0$, satisfies $\Prob(\alpha^\top v = 0) = 0$.
Hence, if Problem \eqref{eq:alg_optimization_pb} admits a solution $\bar{x}$ verifying $\|\proj{(\Vect \cD)^\perp}{\bar{x} - x_0}\|> 0$, then $\alpha^\top \proj{(\Vect \cD)^\perp}{\bar{x} - x_0} \neq 0$ with probability $1$, and therefore
either maximizing or minimizing $\alpha^\top \proj{(\Vect \cD)^\perp}{x^\star - x_0}$ must lead a non-zero objective with probability $1$. This is a linear objective as the projection onto a subspace is linear.
% We will prove this formally in \cref{thm:alg_termination}.

% To address the first challenge, we replace the objective by $\alpha^\top \proj{(\Vect \cD)^\perp}{x^\star - x_0}$ with a randomly generated Gaussian vector $\alpha$. This is now a linear objective since the projection mapping on a vector space is linear, and the randomization trick ensures that if Problem \eqref{eq:alg_optimization_pb} has a solution with non zero objective, then either maximizing or minimizing  $\alpha^\top \proj{(\Vect \cD)^\perp}{x^\star - x_0}$ must lead a non-zero objective with probability $1$. The key remark leading to this result is that for a vector $v$ verifying $\|v\| > 0$, we have $\Prob(\alpha^\top v = 0) = 0$.
% We will prove this formally in \cref{thm:alg_termination}.

\textbf{Linearizing the bilinear, bilevel constraint.}
To address the second challenge, we use complementary slackness conditions, which characterize the optimal solutions of linear programs. We replace $x^\star \in \argmin_{x \in \cX} c^\top x, \; c \in \cC$ by
\begin{align*}
    &x^\star \geq 0, \; s \geq 0, \; \lambda \in \Re^m, \; c \in \cC,\\
    &Ax^\star = b, \; A^\top \lambda + s = c, \; x^\star_is_i = 0, \; \forall i \in [d] 
\end{align*}
The bilinear constraint $x^\star_is_i = 0$ can be linearized by introducing a binary variable $\tau_i \in \{0,1\}$ and adding the constraint
$$
1- \epsilon s_i \geq \tau_i \geq \epsilon x^\star_i
$$
with $\epsilon>0$ a small constant. When \(\mathcal{C}\) is a polyhedron, the resulting formulation is a mixed-integer linear program (MILP) with linear constraints and objectives.

Putting everything together gives Algorithm \ref{alg:LP_case} for linear programs. The algorithm will terminate in exactly $\dim \dir{\dualpoly{\cX}{\cC}}$ iterations. When $\cC$ is a polyhedron, each iteration involves solving a mixed integer program with $O(d+m)$ variables and $O(d+m + \mathrm{constr}(\cC))$ constraints where $\mathrm{constr}(\cC)$ is the number of constraints defining $\cC$. 

\begin{algorithm}[h!]
    \caption{Computing $\dir{\dualpoly{\cX}{\cC}}$}
     \label{algorithm to compute delta}
    \KwIn{Polyhedron $\cX = \{x \geq 0 \; : \; Ax=b\}$, Uncertainty set $\cC$.}
    \KwOut{A basis of $\dir{\dualpoly{\cX}{\cC}}$.}
    Initialize $\cD$ to $\varnothing$.
    
    Set $x_0\in \arg\min_{x\in \cX}c_0^\top x$ for some $c_0 \in \cC$.

    Sample $\alpha \sim \mathcal{N}(0,Id)$.
    
    \textbf{while} either of the problems
    % \begin{align*}
    %     \min&\;\alpha^\top \proj{(\Vect \cD)^\perp}{x_0-x} & \max&\;\alpha^\top \proj{(\Vect \cD)^\perp}{x_0-x}\\
    %     \text{s.t.}&\;x \geq 0,\; \lambda \in \R^m,\; s\in \R_+^d,\; c\in \cC& \text{s.t.} &\;x\geq 0,\; \lambda \in \R^m,\; s\in \R_+^d,\; c\in \cC \\
    %     &Ax=b, \; A^\top \lambda +s=c, & & Ax=b, \; A^\top \lambda +s=c \\
    %     & 1- \epsilon s_i \geq \tau_i \geq \epsilon x_i, \; \tau_i \in \{0,1\}, \; \forall i & & 1- \epsilon s_i \geq \tau_i \geq \epsilon x_i, \; \tau_i \in \{0,1\}, \; \forall i
    % \end{align*}

    \begin{align*}
        \min / \max&\;\alpha^\top \proj{(\Vect \cD)^\perp}{x_0-x}\\
        \text{s.t.}&\;x \geq 0,\; \lambda \in \R^m,\; s\in \R_+^d,\; c\in \cC\\
        &Ax=b, \; A^\top \lambda +s=c, \\
        & 1- \epsilon s_i \geq \tau_i \geq \epsilon x_i, \; \tau_i \in \{0,1\}, \; \forall i
    \end{align*}

    \quad has a solution $x^\star$ with non-zero optimal value,
    
    \quad $\cD \leftarrow \cD \cup \{x^\star - x_0\}$.
    
    \quad resample $\alpha \sim \mathcal{N}(0,Id)$.
    
    \textbf{return} $\cD$
\label{alg:LP_case}
\end{algorithm}

\begin{theorem}[Correctness]\label{thm:alg_termination}
   Algorithm \ref{algorithm to compute delta} terminates with probability $1$ after $\dim \dir{\dualpoly{\cX}{\cC}} \leq d$ steps and outputs a basis of $\dir{\dualpoly{\cX}{\cC}}$.
\end{theorem}

By combining our previous analysis and results, we derive two algorithms that enable optimal decision-making based on prior knowledge of $c$ (uncertainty set $\cC$) and a prescribed measurement of datapoints guided by our theory. Specifically, \cref{alg:data collection algorithm} identifies which elements of $\cQ$ should be queried, using the output of \cref{alg:LP_case}, and returns a set $\cD$ containing these elements. Then, \cref{alg: decision making} produces a decision based on the information obtained by evaluating the objective function at the elements of $\cD$.

\begin{algorithm}[H] 
    \caption{Data Selection Under Query Constraints}
    \label{alg:data collection algorithm}
    \KwIn{Polyhedron $\cX$, Uncertainty Set $\cC$, Query Set $\cQ = \{q_1,\ldots,q_d \}$ (basis of $\R^d$)}
    \KwOut{A minimal \suff\ dataset under constraint $\cD \subset \cQ$.}
    Find $\{v_1,\dots,v_k\}$ a basis of $\spc{\dualpoly{\cX}{\cC}}$ using \cref{alg:LP_case}
    
    $Q \leftarrow [q_1,\ldots,q_d]$
    
    \textbf{return} $\cD:=\{q_i \; : \; i\in [d] \; \text{s.t.} \;\exists j\in [k],\; (Q^{-1}v_j)^\top e_i\neq 0\}$.
\end{algorithm}

\begin{algorithm}[H] 
    \caption{Decision-making with a \sufficient{}}
     \label{alg: decision making}
    \KwIn{Decision set $\cX$, Uncertainty Set $\cC$, \Sufficient\ $\cD = \{q_1,\ldots,q_N\}$, Oracle $\pi$ such that for any $q\in \cQ$, $\pi(q)=c^\top q$ where $c$ is the ground truth.
}
    \KwOut{A decision $\hat{x}\in \arg\min_{x\in \cX}c^\top x$.}

    $o_1,\ldots,o_N \leftarrow \pi(q_1),\ldots,\pi(q_N)$ 

    Compute $\hat{c}\in \argmin \{\sum_{i=1}^{N}(c'^\top q_i - o_i)^2 \; : \; c' \in \cC\}$.

    \textbf{return} $\hat{x} \in \argmin_{x\in \cX}\hat{c}^\top x$.
\end{algorithm}

\section{Application: Hiring Interviews} \label{sec:experiments}

To illustrate our insights, we apply our theoretical framework to the hiring problem detailed in \cref{sec:intro}.
The smallest sufficient dataset here is the smallest subset of candidates to interview to recover the optimal hiring decision. The application illustrates how the task constraints and uncertainty set shape data needs. 
The goal is to hire $20$ candidates from a pool of $d=100$ candidates. Each candidate is associated with two features: GPA and years of experience.
We study two settings: vanilla hiring, with only a total hire cap, and experience-constrained hiring, which also limits hires per seniority group. The decision sets are 
$\cX_{\text{vanilla}} := \{x \in \{0,1\}^d : \sum_{i=1}^d x_i \leq 20\}$ and $\cX_{\text{experience}} := \{x \in \cX_{\text{vanilla}} : \forall j \in [4], \sum_{i \in I_j} x_i \leq 8\},$ where $I_j$ is the set of candidates with $j$ years of experience. 
These constraints are totally unimodular, so relaxing $x \in \{0,1\}^d$ to $x \in [0,1]^d$ still yields optimal solutions via LP \citep[Chapter 3]{wolsey2020integer}. We assume a noisy linear model, i.e. candidate scores belong to
$$
\cC := \{c \in \mathbb{R}^d : \exists \alpha \in \mathbb{R}^2,\; \exists \varepsilon \in [-\eta, \eta],\; \ell \leq \alpha \leq u,\; c = \alpha^\top\phi  + \varepsilon\},
$$
%where $m \in \mathbb{N},\; \eta \in \mathbb{R}_+,\; \ell, u \in \mathbb{R}_+^m$.

%where $m\in \mathbb N$, $\eta \in \R_+$, $\ell,u\in \R_+^m$. In \cref{fig:vanilla hiring} and \cref{fig:diversity hiring}, we plot in red the candidates that should be interviewed to make an optimal decision, and in blue those who do not need to be interviewed. As noise increases, the number of interviews required for an optimal decision also rises: this is coherent with our theory, since the data needed to find the optimal hiring is proportional with $\dim \spc{\dualpoly{\cX}{\cC}}$ where $\cX\in \{\cX_{\text{vanilla}},\cX_{\text{experience}}\}$ (see \cref{cor:dualpoly-carac}), which is increasing when the noise has a higher amplitude. In \cref{fig:vanilla hiring}, candidates naturally fall into three groups based on their prior scores: those with very low scores who are never hired (bottom left, blue), those with very high scores who are always hired (top right, blue), and those in between who require interviews to assess their suitability. This intuitive segmentation emerges directly from our theory. A similar pattern appears in \cref{fig:diversity hiring}, but across majors: at low noise levels, majors are treated separately as candidates from different majors have non-overlapping scores. As noise increases, cross-major comparisons emerge, and candidates begin to "mix", competing across disciplines. Again, this behavior aligns well with our theoretical predictions.

where $\eta \geq 0$ controls the noise level, and $\ell=(4,4), u =(5,5)$. $\phi$ is a feature matrix whose rows are GPAs and years of experience of candidates. 

The GPAs of candidates are generated using a uniform distribution in the interval $[2,4]$, and the level of experience is also uniform in $\{1,2,3,4,5\}$. 
\cref{fig:hiring experiment} indicates candidates to interview to enable an optimal hiring decision, which is the output of \cref{alg:data collection algorithm} with the different sets $\cC$ and $\cX$, and $\cQ$ as the canonical basis vectors.

\paragraph{Impact of $\cC$.} As noise increases (the uncertainty set $\cC$ grows larger, from left to right), so does the number of required interviews: more uncertainty requires more data points. 

\paragraph{Impact of $\cX$.} In the first row of \cref{fig:hiring experiment}, candidates fall into three groups: low scorers (never hired), high scorers (always hired), and mid scorers (interviewed)---an intuitive pattern given the task, automatically recovered by our algorithm. 

When adding group hiring constraints---second row of \cref{fig:hiring experiment}---a similar pattern arise, but now across experience groups rather than the entire population: low noise yields separate treatment between experience groups, as scores don’t overlap across experience levels; high noise leads to cross-experience group comparisons and mixing---again, an intuitive pattern given the new constraints. 

One would naively expect that since $\cX_{\text{experience}}$ is smaller than $\cX_{\text{vanilla}}$, more data would be needed to make optimal decision in the vanilla setting, but that is not necessarily true. In \cref{fig:hiring experiment}, we see that in the high noise regime, more data is needed for the experience-constrained setting than the vanilla setting. In reality, the data needed depends on the geometry of the decision set $\cX$ relative to the uncertainty set $\cC$, as can be seen from \cref{thm: relatively open characterization} and \cref{cor:dualpoly-carac}.

\begin{figure}
    \centering
    \includegraphics[width=1\linewidth]{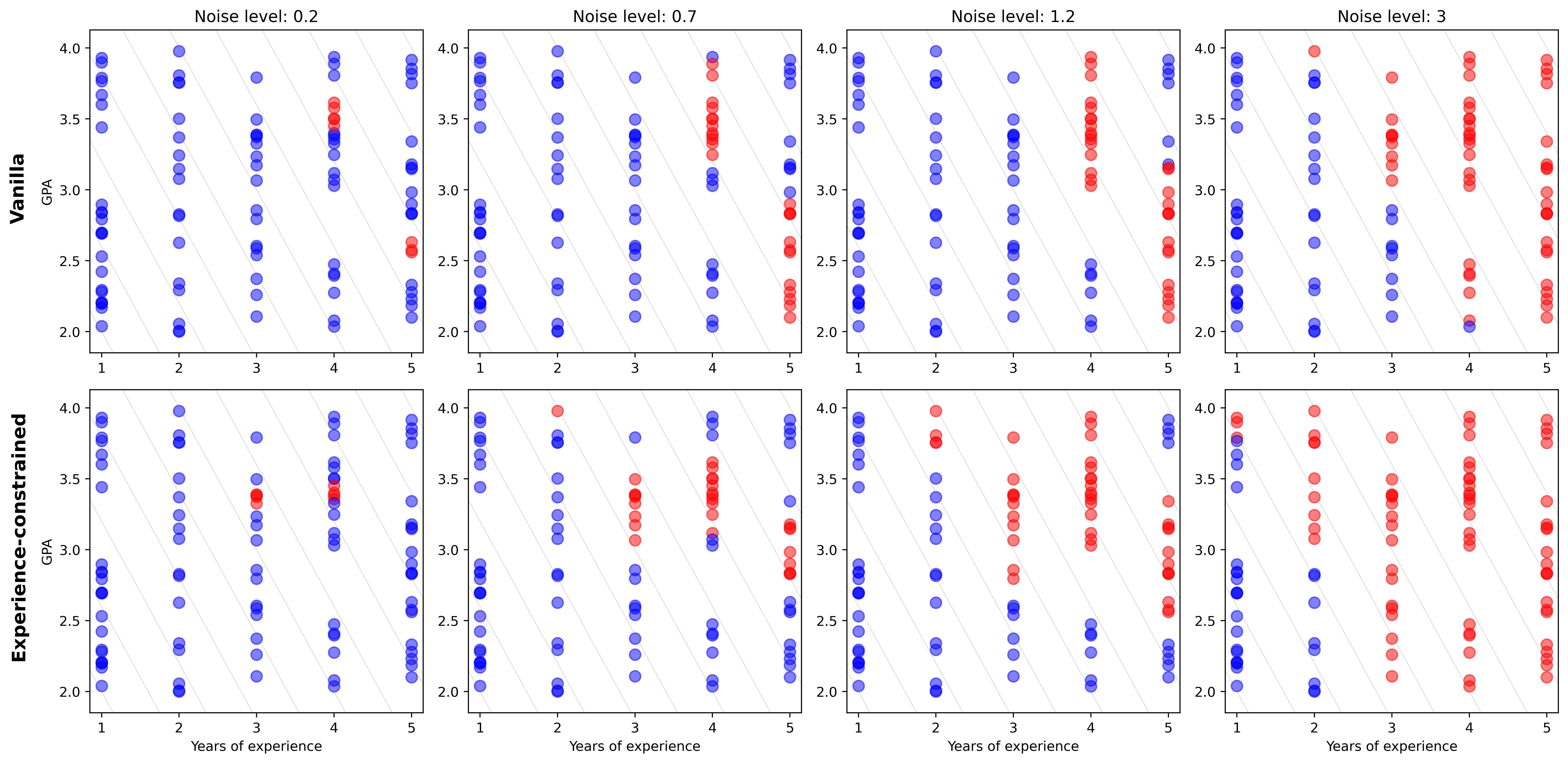}
    \caption{Candidates to be interviewed (in red) to make an optimal hiring decision. Number of candidates to interview from left to right for top and bottom row respectively: $8,24,31,52$ and $8,28,43,70$.}
    \label{fig:hiring experiment}
\end{figure}

\section{Conclusion and Limitations}
This paper introduces a framework for quantifying the informativeness of datasets in decision-making tasks. 
% Our results offer both theoretical insight---characterizing which datasets enable recovery of optimal decisions---and practical tools for constructing minimal sufficient datasets.
While our analysis yields sharp results, several natural extensions remain. First, we restrict attention to linear optimization; extending the framework to other problem classes, such as mixed-integer or convex programs, is an important direction even at the cost of approximate characterizations. Second, we assume query sets are basis vectors; accounting for general queries is a hard problem, but opens rich avenues for exploration. Third, our focus on convex, open uncertainty sets excludes important structured cases such as low-dimensional or discrete sets encoding symmetry or logical constraints. Finally, alternative notions of informativeness, such as approximate rather than exact optimality, or noisy observations rather than clean, merit further study.

% We introduce a principled framework for characterizing the informativeness of datasets relative to specific decision-making tasks, focusing on linear programs under uncertain cost parameters. Through sharp geometric characterizations, we identify the precise conditions under which a fixed dataset suffices to recover the task-optimal decision. 
% These results reveal that informativeness is governed not by global properties of the data or model, but by a low-dimensional subspace of directions along which decisions vary---determined jointly by the uncertainty set and the task structure. 
% Building on this insight, we design an efficient algorithm for constructing minimal sufficient datasets, with provable guarantees and tractable implementation when the uncertainty set is polyhedral. 
% Our framework unifies and extends classical ideas from experimental design and decision theory and enables new forms of task-aware data acquisition across domains such as power systems, supply chains, healthcare, and urban mobility---where decisions hinge on uncertain inputs and only a subset of available data may be needed to recover optimal actions. By formalizing task-specific data value in a finite-sample, non-adaptive setting, our contribution opens the door to connecting geometric sufficiency with statistical notions of information.

\newpage 
\bibliography{references}

\begin{thebibliography}{30}
\providecommand{\natexlab}[1]{#1}
\providecommand{\url}[1]{\texttt{#1}}
\expandafter\ifx\csname urlstyle\endcsname\relax
  \providecommand{\doi}[1]{doi: #1}\else
  \providecommand{\doi}{doi: \begingroup \urlstyle{rm}\Url}\fi

\bibitem[Purohit et~al.(2019)Purohit, Gollapudi, and Raghavan]{purohitHiringUncertainty2019}
Manish Purohit, Sreenivas Gollapudi, and Manish Raghavan.
\newblock Hiring {{Under Uncertainty}}.
\newblock In \emph{Proceedings of the 36th {{International Conference}} on {{Machine Learning}}}, pages 5181--5189. PMLR, May 2019.

\bibitem[Epstein and Ma(2024)]{epsteinSelectionOrdering2024}
Boris Epstein and Will Ma.
\newblock Selection and {{Ordering Policies}} for {{Hiring Pipelines}} via {{Linear Programming}}.
\newblock \emph{Operations Research}, 72\penalty0 (5):\penalty0 2000--2013, September 2024.
\newblock ISSN 0030-364X.
\newblock \doi{10.1287/opre.2023.0061}.

\bibitem[Kleinberg(2005)]{kleinbergMultiplechoicesecretary2005}
Robert Kleinberg.
\newblock A multiple-choice secretary algorithm with applications to online auctions.
\newblock In \emph{Proceedings of the Sixteenth Annual {{ACM-SIAM}} Symposium on {{Discrete}} Algorithms}, {{SODA}} '05, pages 630--631, USA, January 2005. {Society for Industrial and Applied Mathematics}.
\newblock ISBN 978-0-89871-585-9.

\bibitem[Arlotto and Gurvich(2019)]{arlottoUniformlyBounded2019}
Alessandro Arlotto and Itai Gurvich.
\newblock Uniformly {{Bounded Regret}} in the {{Multisecretary Problem}}.
\newblock \emph{Stochastic Systems}, 9\penalty0 (3):\penalty0 231--260, September 2019.
\newblock ISSN 1946-5238.
\newblock \doi{10.1287/stsy.2018.0028}.

\bibitem[Bray(2019)]{brayDoesMultisecretary2019}
Robert Bray.
\newblock Does the {{Multisecretary Problem Always Have Bounded Regret}}?, December 2019.

\bibitem[Settles(2009)]{settlesActiveLearningLiterature2009}
Burr Settles.
\newblock Active {{Learning Literature Survey}}.
\newblock 2009.

\bibitem[Lattimore and Szepesv{\'a}ri(2020)]{lattimoreBanditAlgorithms2020}
Tor Lattimore and Csaba Szepesv{\'a}ri.
\newblock \emph{Bandit {{Algorithms}}}.
\newblock Cambridge University Press, 1 edition, July 2020.
\newblock ISBN 978-1-108-57140-1 978-1-108-48682-8.
\newblock \doi{10.1017/9781108571401}.

\bibitem[Zhao(2024)]{zhaoExperimentalDesign2024a}
Jinglong Zhao.
\newblock Experimental {{Design}} for {{Causal Inference Through}} an {{Optimization Lens}}.
\newblock In \emph{Tutorials in {{Operations Research}}: {{Smarter Decisions}} for a {{Better World}}}, {{INFORMS TutORials}} in {{Operations Research}}, chapter~1, pages 146--188. INFORMS, October 2024.
\newblock ISBN 979-8-9882856-2-5.
\newblock \doi{10.1287/educ.2024.0277}.

\bibitem[Blackwell(1953)]{blackwellEquivalentComparisonsExperiments1953}
David Blackwell.
\newblock Equivalent {Comparisons} of {Experiments}.
\newblock \emph{The Annals of Mathematical Statistics}, 24\penalty0 (2):\penalty0 265--272, 1953.
\newblock ISSN 0003-4851.
\newblock URL \url{https://www.jstor.org/stable/2236332}.
\newblock Publisher: Institute of Mathematical Statistics.

\bibitem[{de Oliveira}(2018)]{deoliveiraBlackwellsinformativeness2018}
Henrique {de Oliveira}.
\newblock Blackwell's informativeness theorem using diagrams.
\newblock \emph{Games and Economic Behavior}, 109:\penalty0 126--131, May 2018.
\newblock ISSN 0899-8256.
\newblock \doi{10.1016/j.geb.2017.12.008}.

\bibitem[Le~Cam(1996)]{lecamComparisonExperimentsShort1996}
L.~Le~Cam.
\newblock Comparison of {Experiments}: {A} {Short} {Review}.
\newblock \emph{Lecture Notes-Monograph Series}, 30:\penalty0 127--138, 1996.
\newblock ISSN 07492170.
\newblock URL \url{http://www.jstor.org/stable/4355942}.
\newblock Publisher: Institute of Mathematical Statistics.

\bibitem[Huber(1992)]{huberRobustEstimation1992}
Peter~J. Huber.
\newblock Robust {{Estimation}} of a {{Location Parameter}}.
\newblock In Samuel Kotz and Norman~L. Johnson, editors, \emph{Breakthroughs in {{Statistics}}: {{Methodology}} and {{Distribution}}}, pages 492--518. Springer New York, New York, NY, 1992.
\newblock ISBN 978-1-4612-4380-9.
\newblock \doi{10.1007/978-1-4612-4380-9_35}.

\bibitem[Hampel et~al.(1986)Hampel, Ronchetti, Rousseeuw, and Stahel]{hampelRobustStatisticsApproach1986}
Frank~R. Hampel, Elevezio~M. Ronchetti, Peter~J. Rousseeuw, and Werner~A. Stahel, editors.
\newblock \emph{Robust Statistics: The Approach Based on Influence Functions}.
\newblock Wiley Series in Probability and Mathematical Statistics {{Probability}} and Mathematical Statistics. Wiley, New York, digital print edition, 1986.
\newblock ISBN 978-0-471-73577-9.

\bibitem[Broderick et~al.(2023)Broderick, Giordano, and Meager]{broderickAutomaticFiniteSampleRobustness2023}
Tamara Broderick, Ryan Giordano, and Rachael Meager.
\newblock An {{Automatic Finite-Sample Robustness Metric}}: {{When Can Dropping}} a {{Little Data Make}} a {{Big Difference}}?, July 2023.

\bibitem[Jiang et~al.(2023)Jiang, Liang, Zou, and Kwon]{jiangOpenDataValUnified2023}
Kevin~Fu Jiang, Weixin Liang, James Zou, and Yongchan Kwon.
\newblock {{OpenDataVal}}: A {{Unified Benchmark}} for {{Data Valuation}}, October 2023.

\bibitem[Ilyas et~al.(2022)Ilyas, Park, Engstrom, Leclerc, and Madry]{ilyasDatamodelsUnderstanding2022}
Andrew Ilyas, Sung~Min Park, Logan Engstrom, Guillaume Leclerc, and Aleksander Madry.
\newblock Datamodels: {{Understanding Predictions}} with {{Data}} and {{Data}} with {{Predictions}}.
\newblock In \emph{Proceedings of the 39th {{International Conference}} on {{Machine Learning}}}, pages 9525--9587. PMLR, June 2022.

\bibitem[Dass et~al.(2025)Dass, Khaddaj, Engstrom, Madry, Ilyas, and {Mart{\'i}n-Mart{\'i}n}]{dassDataMILSelecting2025}
Shivin Dass, Alaa Khaddaj, Logan Engstrom, Aleksander Madry, Andrew Ilyas, and Roberto {Mart{\'i}n-Mart{\'i}n}.
\newblock {{DataMIL}}: {{Selecting Data}} for {{Robot Imitation Learning}} with {{Datamodels}}, May 2025.

\bibitem[Freund and Hopkins(2023)]{freundPracticalRobustnessAuditing2023}
Daniel Freund and Samuel~B. Hopkins.
\newblock Towards {{Practical Robustness Auditing}} for {{Linear Regression}}, July 2023.

\bibitem[Rubinstein and Hopkins(2024)]{rubinsteinRobustnessAuditingLinear2024}
Ittai Rubinstein and Samuel~B. Hopkins.
\newblock Robustness {{Auditing}} for {{Linear Regression}}: {{To Singularity}} and {{Beyond}}, October 2024.

\bibitem[Weitzman(1979)]{weitzmanOptimalSearch1979}
Martin~L. Weitzman.
\newblock Optimal {{Search}} for the {{Best Alternative}}.
\newblock \emph{Econometrica}, 47\penalty0 (3):\penalty0 641--654, 1979.
\newblock ISSN 00129682, 14680262.
\newblock \doi{10.2307/1910412}.

\bibitem[Singla(2018)]{singlaPriceInformationCombinatorial2018}
Sahil Singla.
\newblock The {{Price}} of {{Information}} in {{Combinatorial Optimization}}.
\newblock In \emph{Proceedings of the 2018 {{Annual ACM-SIAM Symposium}} on {{Discrete Algorithms}} ({{SODA}})}, Proceedings, pages 2523--2532. {Society for Industrial and Applied Mathematics}, January 2018.
\newblock \doi{10.1137/1.9781611975031.161}.

\bibitem[Gallego and Segev(2022)]{gallegoConstructiveProphet2022}
Guillermo Gallego and Danny Segev.
\newblock A {{Constructive Prophet Inequality Approach}} to {{The Adaptive ProbeMax Problem}}, October 2022.

\bibitem[Chaloner and Verdinelli(1995)]{chalonerBayesianExperimental1995}
Kathryn Chaloner and Isabella Verdinelli.
\newblock Bayesian {{Experimental Design}}: {{A Review}}.
\newblock \emph{Statistical Science}, 10\penalty0 (3):\penalty0 273--304, 1995.
\newblock ISSN 0883-4237.

\bibitem[Singh and Xie(2020)]{singhApproximationAlgorithms2020}
Mohit Singh and Weijun Xie.
\newblock Approximation {{Algorithms}} for {{D-optimal Design}}.
\newblock \emph{Mathematics of Operations Research}, 45\penalty0 (4):\penalty0 1512--1534, November 2020.
\newblock ISSN 0364-765X.
\newblock \doi{10.1287/moor.2019.1041}.

\bibitem[Mitzenmacher and Vassilvitskii(2020)]{mitzenmacherAlgorithmsPredictions2020}
Michael Mitzenmacher and Sergei Vassilvitskii.
\newblock Algorithms with {{Predictions}}, June 2020.

\bibitem[Blackwell(1949)]{blackwellComparisonreconnaissances1949}
David Blackwell.
\newblock \emph{Comparison of Reconnaissances}.
\newblock RAND Corporation, Santa Monica, CA, 1949.

\bibitem[Blackwell(1951)]{blackwellComparisonExperiments1951}
David Blackwell.
\newblock Comparison of {{Experiments}}.
\newblock In Jerzy Neyman, editor, \emph{Proceedings of the {{Second Berkeley Symposium}} on {{Mathematical Statistics}} and {{Probability}}}, pages 93--102. University of California Press, 1951.
\newblock ISBN 978-0-520-41158-6.
\newblock \doi{10.1525/9780520411586-009}.

\bibitem[Sherman(1951)]{shermanTheoremHardy1951}
S.~Sherman.
\newblock On a {{Theorem}} of {{Hardy}}, {{Littlewood}}, {{Polya}}, and {{Blackwell}}.
\newblock \emph{Proceedings of the National Academy of Sciences of the United States of America}, 37\penalty0 (12):\penalty0 826--831, 1951.
\newblock ISSN 00278424, 10916490.

\bibitem[Boll(1955)]{boll_comparison_1955}
Charles Boll.
\newblock \emph{Comparison of experiments in the infinite case and the use of invariance in establishing suffici}.
\newblock PhD thesis, Stanford, 1955.

\bibitem[Wolsey(2020)]{wolsey2020integer}
Laurence~A Wolsey.
\newblock \emph{Integer programming}.
\newblock John Wiley \& Sons, 2020.

\end{thebibliography}

%%%%%%%%%%%%%%%%%%%%%%%%%%%%%%%%%%%%%%%%%%%%%%%%%%%%%%%%%%%%

\newpage
\appendix

\section{Proofs}

\subsection{Proof of Proposition \ref{prop:sufficient:projections}}
\begin{proof} \;
    \begin{itemize}
        \item
        $(\Rightarrow)$ Assume that $\cD$ is a \sufficient{}. Let $c,c'\in \cC$ such that $c_{\Vect \cD}=c'_{\Vect \cD}$. We have for any $q\in \cD$, $c^\top q=c'^\top q$. Let $\hat{X}$ given by Definition \ref{def:sufficient}. We have $\hat X\left(c^\top q_1,\dots,c^\top q_N\right)=\hat X\left(c'^\top q_1,\dots,c'^\top q_N\right)$ i.e. $\arg\min_{x\in \cX}c^\top x=\arg\min_{x\in \cX}c'^\top x$.
        
        \item $(\Leftarrow)$ Assume that $\cD$ satisfies the property of the proposition. Since for any $c,c'\in \cC$ we have 
        $
            c_{\Vect \cD}=c'_{\Vect \cD}\Longleftrightarrow (c^\top q)_{q\in \cD}=(c'^\top q)_{q\in \cD},
        $
        then for any $c\in \cC$, we define $\hat X\left(c^\top q_1,\dots,c^\top q_N\right)$ to be equal to $\arg\min_{x\in \cX}c'^\top x$ for any $c'$ such that $c'_{\Vect \cD}=c_{\Vect \cD}$. This mapping is well-defined and verifies the desired property.
    \end{itemize}
\end{proof}
\subsection{Proof of Proposition \ref{prop:noisy approximation}}
\begin{proof}
Let $Q$ be a matrix whose rows are the elements of $\cD$ and $\hat{c}(o_1,\dots,o_r)\in \argmin \{\sum_{i=1}^{r}(c'^\top q_i - o_i)^2 \; : \; c' \in \cC\}$. Let $\eta := Q\hat{c}(o_1,\dots,o_r) - Q \ctrue$. Since $\hat{c}(o_1,\dots,o_r)=\arg\min_{c\in \cC}\norm{Qc-o}$, then we have $\norm{Q\hat{c}(o_1,\dots,o_r) -Q\ctrue - \varepsilon}\leq \norm{\varepsilon}$. Hence, we have  $\norm{Q\hat{c}(o_1,\dots,o_r) -Q\ctrue}\leq 2\norm{\varepsilon}$ and consequently $\norm{\eta}\leq 2 \norm{\varepsilon}$. 
    We would like to show that the distance between the projections of $\ctrue$ and $\hat{c}(o_1,\dots,o_r)$ in the span of $\cD$ is upper bounded by $O(\norm{\varepsilon})$. Consider $\alpha_{\text{true}},\hat{\alpha}\in \R^r$ such that $\hat{c}(o_1,\dots,o_r)_{\Vect \cD}=Q^\top \hat{\alpha}$ and $c_{\text{true},\Vect \cD}=Q^\top \alpha_{\text{true}}$. Without loss of generality, we can assume that $\cD$ is linearly independent. Indeed, if $\cD$ was linearly dependent, it would provide exactly the same information as any smallest cardinality subset of $\cD$ that spans all elements of $\cD$. In this case $Q$ is full rank and $QQ^\top$ is invertible. We have 
    \begin{align*}
        Q\hat{c}(o_1,\dots,o_r)-Q\ctrue=\eta &\implies Q(\hat{c}(o_1,\dots,o_r)_{\Vect \cD} - c_{\text{true},\Vect \cD})=\eta\\
        &\implies QQ^\top (\hat{\alpha}-\alpha_{\text{true}})=\eta\\
        &\implies \hat{\alpha}-\alpha_{\text{true}}=(QQ^\top)^{-1}\eta\\
        &\implies \hat{c}(o_1,\dots,o_r)_{\Vect \cD}-c_{\text{true},\Vect \cD}=Q^\top (QQ^\top)^{-1}\eta.
    \end{align*}
    Let $U\in \R^{r\times r}$, $V\in \R^{d\times d}$ and $\Sigma\in \R^{r\times d}$ such that $U,V$ are orthogonal matrices and for all $(i,j)\in [r]\times [d]$
    \begin{align*}
        \Sigma_{ij}=\begin{cases}\text{$\sigma_i$ the $i-$th singular value of $Q$} & \text{if }i=j\\ 0 &\text{else,}\end{cases}
    \end{align*}
    and $Q=U\Sigma V^\top$. We have 
    \begin{align*}
        Q^\top (QQ^\top)^{-1}&=U\Sigma V^\top (U\Sigma V^\top V \Sigma^\top U^\top)^{-1}\\
        &=U\Sigma V^\top (U\Sigma \Sigma^\top U^\top)^{-1}\\
        &=V\Sigma^\top  U^\top U(\Sigma \Sigma^\top)^{-1} U^\top\\
        &=V\Sigma^\top (\Sigma \Sigma^\top)^{-1} U^\top=V\Sigma'U^\top,
    \end{align*}
    where $\Sigma'\in \R^{d\times r}$ satisfies
    \begin{align*}
        \Sigma'_{ij}=\begin{cases}\text{$\frac{1}{\sigma_i}$ the $i-$th singular value of $Q$} & \text{if }i=j\\ 0 &\text{else.}\end{cases}
    \end{align*}
    Let $\lambda_{\text{min}}(D)$ the smallest singular value of $Q$. The calculations above gives, when $\norm{.}$ is the $L^2$ norm,
    \begin{align*}
        \norm{c_{\text{true},\Vect \cD}-\hat{c}(o_1,\dots,o_r)_{\Vect \cD}}=\norm{Q^\top (QQ^\top)^{-1}\eta}\leq \norm{Q^\top (QQ^\top)^{-1}}\cdot \norm{\eta}\leq \frac{2}{\lambda_{\text{min}}(D)}\norm{\varepsilon}.
    \end{align*}
    We now provide an essential lemma.
\begin{lemma}
        Assume that $\cC$ is open. Let $\cD$ a \sufficient{} for $\cC\subset \R^d$. Let $c\in \cC$. There exists $\mu>0$ such that for any $c'\in \cC$ such that $\norm{c_{\Vect \cD}-c'_{\Vect \cD}}<\mu$, we have $\arg\min_{x\in \cX}c'^\top x \subset \arg\min_{x\in \cX}c^\top x$.
    \end{lemma}
    \begin{proof}
        We assume without loss of generality that $\cC$ is compact (it suffices to replace $\cC$ by some closed ball of small radius centered around $c$ that is a subset of $\cC$). Assume that the result does not hold, i.e. there exists a sequence $c'_n\in \cC$ such that $c'_{n,\Vect D}$ converges to $c_{\Vect D}$, and for all $n\in \mathbb N$, there exists $x'\in \cX^\angle$ such that $x'\in \arg\min_{x\in \cX}c'^\top_n x \setminus \arg\min_{x\in \cX}c^\top x $. Since the number of extreme points in $\cX$ are finite, there exists $x'\in \cX^\angle$ and a strictly increasing map $\varphi:\mathbb N \longmapsto \mathbb N$ such that for all $n\in \mathbb N$, we have $x'\in \arg\min_{x\in \cX}c'^\top_{\varphi(n)} x \setminus \arg\min_{x\in \cX}c^\top x $. Since $\cC$ is compact, we can assume without loss of generality that the sequence $c'_{\varphi(n)}$ is convergent to some $c'\in \cC$ (it suffices to extract another time a converging sequence from $c'_{\varphi(n)}$). Consequently, since for all $n\in \mathbb N$, $c'_{\varphi(n)}\in \Lambda(x')$ (see \cref{optimality cone def} for definition of $\Lambda(x')$), and $\Lambda(x')$ is closed, then $c'\in \Lambda(x')$. Furthermore, we have $c\not\in \Lambda(x)$, and $c_{\Vect \cD}=c'_{\Vect \cD}$, which means that $\arg\min_{x\in \cX}c^\top x=\arg\min_{x\in \cX}c'^\top x$. This implies that $x'\in \arg\min_{x\in \cX}c^\top x$, i.e. $c\in \Lambda(x)$ which is impossible.
    \end{proof}
    
    When $\norm{\varepsilon}\leq \frac{\mu \lambda_{\text{min}}(D)}{2}$, we have 
    \begin{align*}
        \norm{c_{\text{true},\Vect \cD}-\hat{c}(o_1,\dots,o_r)_{\Vect \cD}}<\mu,
    \end{align*}
    i.e. from the lemma above, $\arg\min_{x\in \cX}\hat{c}(o_1,\dots,o_r)^\top x \subset \arg\min_{x\in \cX}\ctrue^\top x$.

    \end{proof}
 \subsection{Proof of Proposition \ref{prop:suff:vectorspace}} \label{proof prop:suff:vectorspace}
        \begin{proof}
        We denote $F:=\Vect \cD$. The condition $F_0\cap \Ker A \subset  \Vect \cD$ is equivalent to $F_0\cap \Ker A\perp F^\perp$, so in order to prove the equivalence with the 3rd proposition, we will prove the equivalence with $F_0\cap \Ker A\perp  F^\perp$.
    \begin{itemize}
        \item Assume that $ F^\perp\perp F_0 \cap \Ker A$.
        Let $c,c'\in \R^d$ such that $c_F=c'_F$. We will show that they have the same $\argmin$, which proves \suff{} as a result of \cref{prop:sufficient:projections}. We show that the mapping $x \in \cX \to (c-c')^\top x$ is constant. In fact, for $x,x' \in \cX$, we have
        \begin{align*} 
            (c-c')^\top x - (c-c')^\top x' = \underbrace{(c-c')^\top}_{\in  F^\perp} \underbrace{(x-x')}_{\in F_0 \cap \Ker A}=0,
            % \label{constant map}
        \end{align*}
        by the assumption $F^\perp\perp F_0 \cap \Ker A$.
Hence, the mappings $x\longmapsto c^\top x$ and $x\longmapsto c'^\top x$ are identical in $\cX$, within a constant. Consequently, we have $\arg \min_{x\in \cX}c^\top x = \arg\min_{x\in \cX}c'^\top x$.
        
        \item Assume that $ F^\perp\not\perp F_0 \cap \Ker A $.
        Let $c\in \R^d$.
        We would like to show that there exists $c'\in \cC$ such that $c_F=c_F'$ and $\arg\min_{x\in \cX}c^\top x \neq \arg\min_{x\in \cX}c'^\top x$. 
      Let $x^\star(c)\in \arg\min_{x\in \cX}c^\top x$. There exists a set of feasible directions for $x^\star(c)$, $V=\{\delta_1,\dots,\delta_r\} \subset \FD(x^\star(c))$, that spans $F_0 \cap \Ker A$ (see \cref{lemma:feasible directions in general polyhedron}). Since $V$ spans $F_0\cap \Ker A$, and $ F^\perp\not\perp F_0 \cap \Ker A $, then there exists $\delta\in V$ such that $\proj{ F^\perp}{\delta}\neq 0$. Let $M$ be a positive constant and define $c' = c -M \proj{ F^\perp}{\delta}$. We have $c'_F = c_F$. For all $\alpha >0$ such that $x^\star(c)+\alpha \delta \in \cX$, we have
    \begin{align*}
                c'^\top (x^\star(c)+\alpha \delta) &= {c'}^\top x^\star(c) + \alpha c^\top \delta - \alpha M\proj{F^\perp}{\delta}^\top \delta\\
                &=c'^\top x^\star(c) + \alpha c^\top \delta - \alpha M \norm{\proj{ F^\perp}{\delta}}^2.
            \end{align*}
            When $M$ is set to be large enough, we can see that we have $c'^\top (x^\star(c)+\alpha \delta )<c'^\top x^\star(c)$, which means that $x^\star (c)\not\in \arg\min_{x\in \cX}c'^\top x$.
  
  \end{itemize}

  % Since $\cC\cap F^\perp\subset \cC$, $\cC'\subset \cC$, $\cC$ is a vector space,  $c'_F=a=c'_{\cC'}$, and $\cC=\cC'\oplus \cC\cap F^\perp$, the condition $c'\in \cC$ allows $c'_{\cC\cap F^\perp}$ to take any value in $\cC\cap F^\perp$ regardless of the value of $a$. 
  
  % Hence, we set $c'_{\cC\cap F^\perp}=-M \proj{\cC\cap F^\perp}{\delta}$ where $M$ is a positive constant. This gives for any $\alpha >0$ such that $x^\star(c)+\alpha \delta \in \cX$,
  %       \begin{align*}
  %           c'^\top (x^\star(c)+\alpha \delta) &= {c'}^\top x^\star(c) + \alpha {c'}_{\cC'}^\top \delta + \alpha {c'}_{\cC\cap F^\perp}^\top \delta\\&={c'}^\top x^\star(c) + \alpha {c'}_{\cC'}^\top \delta - M\alpha \proj{\cC\cap F^\perp}{\delta}^\top \delta\\
  %           &=c'^\top x^\star(c) + \alpha c'^\top \delta -M \alpha \norm{\proj{\cC\cap F^\perp}{\delta}}^2.
  %       \end{align*}
  %       When $M$ is set to be large enough, we can see that we have $c'^\top (x^\star(c)+\alpha \delta )<c'^\top x^\star(c)$, which means that $x^\star (c)\not\in \arg\min_{x\in \cX}c'^\top x$.
  %   \end{itemize}

Let us now prove the final part of the proposition. Let $K>0$ and $\hat{x}:\R^N \longrightarrow \cX$. We first prove that the set of feasible directions from $\hat{x}\left(c^\top q_1,\dots,c^\top q_N\right)$ spans $F_0\cap \Ker A$. We know from Lemma \ref{lemma:feasible directions in general polyhedron} that the set of feasible directions from any extreme point spans $F_0\cap \Ker A $. Let $x^1,\dots,x^\ell$, $\ell\in \mathbb N$ and $\lambda_1,\dots,\lambda_l\in (0,1]$ such that $\hat{x}\left(c^\top q_1,\dots,c^\top q_N\right)=\sum_{i=1}^{\ell}\lambda_i x_i$. For any feasible direction $\delta$ from $x^1$, for $\alpha>0$, we have 
\begin{align*}
    \hat{x}\left(c^\top q_1,\dots,c^\top q_N\right) +\alpha \delta = \lambda_1(x^1 + \frac{1}{\lambda_1}\alpha \delta) + \sum_{i=2}^{\ell}\lambda_ix^i.
\end{align*}
For $\alpha$ small enough, we can see that $x^1 + \frac{1}{\lambda_1}\alpha \delta\in \cX$ and consequently $\hat{x}\left(c^\top q_1,\dots,c^\top q_N\right) +\alpha \delta\in \cX$. Hence, any feasible direction from $x^1$ is feasible from $\hat{x}\left(c^\top q_1,\dots,c^\top q_N\right)$. This means that the feasible directions from $\hat{x}\left(c^\top q_1,\dots,c^\top q_N\right)$ span $F_0\cap \Ker A$. Hence, there exist $\delta\neq 0$ a feasible direction from $\hat{x}\left(c^\top q_1,\dots,c^\top q_N\right)$ such that $\delta_{F^\perp}\neq 0$. $c_{F^\perp }$ can take any value in $ F^\perp $ without changing the values of $c^\top q_1,\dots, c^\top q_N$. Consequently, we set $c_{ F^\perp }=-M\delta_{F^\perp }$ where $M$ is a nonnegative number that we will set later. Hence, letting $\alpha >0$ such that $\hat{x}\left(c^\top q_1,\dots,c^\top q_N\right)+\alpha \delta \in \cX$, we have 
\begin{align*}
    c^\top (\hat{x}\left(c^\top q_1,\dots,c^\top q_N\right)+\alpha \delta)&=c^\top \hat{x}\left(c^\top q_1,\dots,c^\top q_N\right) + \alpha c_F^\top \delta + \alpha c_{ F^\perp }^\top  \delta\\
    &=c^\top \hat{x}\left(c^\top q_1,\dots,c^\top q_N\right) + \alpha c_F^\top \delta  - M \alpha \norm{\delta_{ F^\perp }}^2
\end{align*}
This implies $c^\top \hat{x}\left(c^\top q_1,\dots,c^\top q_N\right) + \alpha c_F^\top \delta  - M \alpha \norm{\delta_{ F^\perp }}^2 \geq \min_{x\in \cX}c^\top x$, i.e. 
\begin{align*}
    c^\top \hat{x}\left(c^\top q_1,\dots,c^\top q_N\right) \geq - \alpha c_F^\top \delta  + M \alpha \norm{\delta_{F^\perp }}^2 +\min_{x\in \cX}c^\top x.
\end{align*}
Taking $M\geq \frac{K +\alpha c_F^\top \delta }{\alpha \norm{\delta_{ F^\perp }}^2 }$, we indeed get 
\begin{align*}
    c^\top \hat{x}\left(c^\top q_1,\dots,c^\top q_N\right) \geq K +\min_{x\in \cX}c^\top x.
\end{align*}
 \end{proof}
\subsection{Proof of Proposition \ref{feasible directions polyhedral cone}} \label{feasible directions polyhedral cone proof}
\begin{proof}

    Let $x^\star \in \cX^\angle$. We denote $J=\{i\in [d], x^\star_i = 0\}$ and $I_0=\{i\in [d],\; \exists x\in \cX,\; x_i\neq 0\}$. For every $\delta \in \R^d$, we have 
    \begin{align*}
        \delta \in \FD(x^\star) \Longleftrightarrow \exists \varepsilon>0,\; x^\star +\varepsilon \delta \geq 0\text{ and }A\delta =0
        \Longleftrightarrow A\delta = 0 \text{ and }\delta_j \geq 0 \text{ for every }j\in J.
    \end{align*}
    This means that $\FD(x^\star)$ is a polyhedral cone, and $\FD(x^\star)\subset \Ker A$. Furthermore, since $[d]\setminus I_0 \subset J$, we also have $\FD(x^\star) \subset F_0$ which yields $\FD(x^\star)\subset F_0 \cap \Ker A$.
\end{proof}
\subsection{Proof of Proposition \ref{optimality cone def}} \label{proof optimality cone def}
\begin{proof}
Let $x^\star \in \cX^\angle$. For every $c\in \R^d$, we have
$$x^\star \in \arg\min_{x\in \cX}c^\top x \Longleftrightarrow \forall \delta \in \FD(x^\star),\; c^\top \delta \geq 0 \Longleftrightarrow \forall \delta \in D(x^\star),\; c^\top \delta \geq 0. $$
\end{proof}

  \subsection{Proof of Theorem \ref{thm:span delta is dir x}} \label{proof thm:span delta is dir x}
 Before proving the theorem, we will have to introduce a few lemmas and definition.
\begin{lemma} \label{pseudo continuity of argmin}
    For any $c\in \R^d$, there exists $\varepsilon>0$ such that for any $c'$ satisfying $\norm{c-c'}<\varepsilon$, $\arg\min_{x\in \cX}c^\top x \cap \arg\min_{x\in \cX}c'^\top x\neq \varnothing$.
\end{lemma}
\begin{proof}
    Assume that there exists $c\in \R^d$ such that for all $\varepsilon>0$, there exists $c'$ satisfying $\norm{c-c'}<\varepsilon$ and $\arg\min_{x\in \cX}c^\top x \cap \arg\min_{x\in \cX}c'^\top x =\varnothing$. There exists a sequence $(c'_n)_{n \in \integ}$ that converges to $c$ such that for all $n\in \mathbb N$, there exists $x\in \cX^\angle\setminus \arg\min_{x\in \cX}c^\top x$ such that $x\in \arg\min_{x\in \cX}c_n'^\top x$. Since there is a finite number of extreme points, there exists a subsequence $(c'_{\varphi(n)})_{n \in \integ}$ and $x\in \cX^\angle\setminus \arg\min_{y\in \cX}c^\top y$ such that for all $n\in \mathbb N$, we have $x\in \arg\min_{y\in \cX}c_{\varphi(n)}'^\top y$, i.e. $c'_{\varphi(n)}\in \Lambda(x)$. Hence, since $\Lambda(x)$ is closed, we have $c\in \Lambda(x)$ and $x\not\in \arg\min_{y\in \cX}c^\top y$ which is not possible.
\end{proof}

\begin{definition}[Extreme Point Neighbors]
    Let $\cC\subset \R^d$. For any two extreme points $x_1,x_2\in X^\angle$, we say that $x_1$ and $x_2$ are neighbors in $\cX$ if there exists an extreme direction $\delta \in D(x_1)$ such that $x_2=x_1+\delta$. We say that they are $\cC-$strong neighbors in $\cX$ if furthermore there exists $c\in \cC$ such that $x,x'\in \arg\min_{y\in \cX}c^\top y$.
\end{definition}

\begin{definition}[Connected and $\cC-$strongly Connected Points]
    For any subset $\cal Y\subset \cX^\angle$ and $\cC\subset \R^d$, for any pair of elements $x,x'\in \cal Y$, we say that $x,x'$ are connected by neighboring extreme points in $\cal Y$ if there exist $h\in \mathbb N$ and a sequence $x_1,\dots,x_h\in  \cal Y$ such that for all $i\in [h-1]$, $x_i$ and $x_{i+1}$ are neighbors in $\cX$ and $x_1=x$ and $x_h=x'$. We say that they are $\cC-$strongly connected, when $x_i$ and $x_{i+1}$ are $\cC-$strong neighbors.
    When there is no ambiguity, we say that $x$ and $x'$ are (strongly) connected. 
    
    We say that the set $\cal Y$ is ($\cC-$strongly) connected by neighboring extreme points if this property holds for any pair of extreme points in $\cal Y$. When there is no ambiguity, we say that $\cal Y$ is (strongly) connected. For any element $x$ of $\cal Y$, we call the ($\cC-$strong) connection class of $x$ the set of points in $\cal Y$ that are ($\cC-$strongly) connected by neighboring extreme points to $x$.
\end{definition}
\begin{lemma} \label{strong connection in argmin}
    For any $c\in \cC$, $\cX^\angle \cap \arg\min_{x\in \cX}c^\top x$ is $\cC-$strongly connected by neighboring extreme points in $\cX$.
\end{lemma}
\begin{proof}
    Let $c\in \cC$. Every extreme point in $\arg\min_{x\in \cX}c^\top x$ is also an extreme point in $\cX$ (see \cref{cone equivalence}), and every extreme direction in $\arg\min_{x\in \cX}c^\top x$ is also an extreme direction in $\cX$. Hence, since $\arg\min_{x\in \cX}c^\top x$ is a bounded polyhedron, $\cX^\angle \cap \arg\min_{x\in \cX}c^\top x$ is connected by neighboring extreme points in $\cX$. Furthermore, by definition, since $\cX^\angle \cap \arg\min_{x\in \cX}c^\top x\subset \arg\min_{x\in \cX}c^\top x$, then $\cX^\angle \cap \arg\min_{x\in \cX}c^\top x$ is $\cC-$strongly connected.
\end{proof}
\begin{lemma} \label{dual poly is strongly connected}
     When $\cC$ is convex, $\dualpoly{\cX}{\cC}\cap \cX^\angle$ is $\cC-$strongly connected by neighboring extreme points.
\end{lemma}
\begin{proof}
    Assume that there exist $x,x'\in  \dualpoly{\cX}{\cC}\cap \cX^\angle$ that are not strongly connected. Let $c,c'\in \cC$ such that $x\in \arg\min_{y\in \cX}c^\top y$, $x'\in \arg\min_{y\in \cX}c'^\top y$. For any $\alpha \in [0,1]$, we denote $c_\alpha:=(1-\alpha)c+\alpha c'$. Let 
    \begin{align*}
        U:=\{x^\star \in \cX^\angle,\; \exists \alpha\in[0,1],\; x^\star \in \arg\min_{y\in \cX}c_\alpha^\top y\}\subset \dualpoly{\cX}{\cC}\cap \cX^\angle.
    \end{align*}
    
    Let $K$ be the intersection of $U$ and the connection class of $x$. We have $x'\not\in K$. Let
    
    \begin{align*}
        \alpha^\star =\max\left\{\alpha \in [0,1],\;  K \cap \arg\min_{y\in \cX}c_\alpha^\top y \neq \varnothing\right\}.
    \end{align*}
    
    If $\alpha^\star=1$, then there exists $v\in \arg\min_{y\in \cX}c'^\top y$ such that $v\in K$. From Lemma \ref{strong connection in argmin}, $\cX^\angle \cap \arg\min_{y\in \cX}c'^\top y$ is $\cC-$strongly connected and $v\in K\cap \cX^\angle \cap \arg\min_{y\in \cX}c'^\top y$ and consequently $x'\in K$, and therefore is connected to $x$ which contradicts our assumption. Hence, we necessarily have $\alpha^\star<1$. 
    Furthermore, from Lemma \ref{pseudo continuity of argmin}, there exists $\varepsilon\in (0,1-\alpha^\star)$ such that 

    \begin{equation}\label{eq:intersec_argmins}
    \arg\min_{y\in \cX}c_{\alpha^\star + \varepsilon}^\top y \cap \arg\min_{y\in \cX}c_{\alpha^\star}^\top y \neq \varnothing.
    \end{equation}

    As $K\cap \arg\min_{y\in \cX}c_{\alpha^\star}^\top y \neq \varnothing$ and $\arg\min_{y\in \cX}c_{\alpha^\star}^\top y $ is $\cC-$strongly connected from \cref{strong connection in argmin}, we have $\arg\min_{y\in \cX}c_{\alpha^\star}^\top y \subset K$. Combined with \eqref{eq:intersec_argmins}, it implies that $K\cap \arg\min_{y\in \cX}c_{\alpha^\star+\varepsilon}^\top y \neq \varnothing$. This contradicts the supremum definition of $\alpha^\star$.
\end{proof}

We have now enough tools to prove the theorem.

\begin{proof}[Proof of \cref{thm:span delta is dir x}]

We have
    \begin{align} 
        \Vect \Delta(\cX,\cC)& \underset{(1)}{=} \Vect \{x_1-x_2,\; x_1,x_2\in \cX^\angle\cap \dualpoly{\cX}{\cC},\; x_1 \text{ and }x_2 \text{ are $\cC-$strong neighbors}\} \label{set 1}
        \\& \underset{(2)}{=} \Vect \{x_1-x_2,\; x_1,x_2\in \cX^\angle\cap \dualpoly{\cX}{\cC}\} \label{set 2}
        \\& \underset{(3)}{=}\dir{\cX^\angle\cap \dualpoly{\cX}{\cC}} \label{set 3}
        \\&\underset{(4)}{=}\dir{\dualpoly{\cX}{\cC}}.\label{set 5}
    \end{align}
    Let's justify each of the equalities above.
    \begin{itemize}
        \item (1) Let $\delta \in \Delta(\cX,\cC)$. There exists $c\in \cC$ and $x\in \cX^\angle$ such that $c\in F(x,\delta)$. This means that $x\in \arg\min_{y\in \cX}c^\top y$, $\delta$ is an extreme direction for $x$ in $\cX$, and $c^\top \delta=0$. Consequently, there exists $\eta>0$ such that $x':=x+\eta \delta $ is an extreme point%justify?
        , that is a neighbor of $x$ by definition. Also, we have $x'\in \arg\min_{y\in \cX}c^\top y$. Hence, $\delta =\frac{1}{\eta}(x'-x)$, which proves
        $$\Delta(\cX,\cC)\subset \Vect\{x_1-x_2,\; x_1,x_2\in \cX^\angle\cap \dualpoly{\cX}{\cC},\; x_1 \text{ and }x_2 \text{ are $\cC-$strong neighbors}\}.$$
        Conversely, if $x_1,x_2\in \cX^\angle \cap \dualpoly{\cX}{\cC}$ are $\cC-$strong neighbors, then there exists an extreme direction $\delta$ for $x_1$ such that $x_2=x_1 + \delta$ and $c\in \cC$ such that $x_1,x_2\in \arg\min_{y\in \cX}c^\top y$. Hence, we have $c^\top \delta=0$, and consequently $c\in F(x_1,\delta)$, which means that $\delta \in \Delta(\cX,\cC)$, i.e. $x_1-x_2 \in \Delta (\cX,\cC)$. This proves the desired equality.
        \item (2) Set \cref{set 1} is clearly a subset of set \cref{set 2}. Let's prove the converse inclusion. Let $x,x'\in \cX^\angle \cap \dualpoly{\cX}{\cC}$. According to Lemma \ref{dual poly is strongly connected}, there exists $h\in \mathbb N$ and a sequence $x_1,\dots,x_h\in \cX^\angle \cap \dualpoly{\cX}{\cC}$ such that $x_1=x$ and $x_h=x'$ and for all $i\in [h-1]$, $x_i,x_{i+1}$ are $\cC-$strongly connected. Hence, we have 
        \begin{align*}
            x-x'=\sum_{i=1}^{h-1}x_{i+1}-x_i.
        \end{align*}
        All of the terms in the sum above are in set $\eqref{set 1}$ and therefore their sum as well, by linearity. Hence, we indeed have the inclusion.
        \item (3) This equality is immediate since for any $x_1,x_2 \in \cX^\angle\cap \dualpoly{\cX}{\cC}$, $x_1-x_2=x_1 - x_0 -(x_2-x_0)$ for any $x_0\in  \cX^\angle\cap \dualpoly{\cX}{\cC}$ and consequently $x_1-x_2\in \dir{ \cX^\angle\cap \dualpoly{\cX}{\cC}}$. 
        %\item (4) Let $x_0\in  \cX^\angle\cap \dualpoly{\cX}{\cC}$. We have 
        %\begin{align*}
        %    \conv{ \cX^\angle\cap \dualpoly{\cX}{\cC} - x_0} &\subset \Vect( \cX^\angle\cap \dualpoly{\cX}{\cC} - x_0)\\
        %    &\subset \Vect \left( \conv{\cX^\angle\cap \dualpoly{\cX}{\cC} - x_0}\right)\\
        %    &=\Vect \left( \conv{\cX^\angle\cap \dualpoly{\cX}{\cC} }- x_0\right)\\
        %    &=\dir{\conv{\cX^\angle\cap \dualpoly{\cX}{\cC} }}.
        %\end{align*}
        %Consequently, we indeed have $\dir{\cX^\angle\cap \dualpoly{\cX}{\cC} }=\dir{\conv{\cX^\angle\cap \dualpoly{\cX}{\cC} }}$.
        \item (4) In order to prove this equality, we prove that $\dualpoly{\cX}{\cC}\subset \conv(\cX^\angle \cap \dualpoly{\cX}{\cC})$. Let $x\in \dualpoly{\cX}{\cC}$ and $c\in \cC$ such that $x\in \arg\min_{y\in \cX}c^\top y$. There exists $\alpha_1,\dots,\alpha_k\in (0,1]$ such that $\alpha_1+\dots+\alpha_k=1$ and $x_1,\dots,x_k\in \cX^\angle$ such that $x=\sum_{i=1}^{k}\alpha_kx_k$. We have 
        \begin{align*}
            \min_{y\in \cX}c^\top y \geq \sum_{i=1}^{k}\alpha_kc^\top x_k \text{ i.e. }\sum_{i=1}^{k}\alpha_k(c^\top x_k-\min_{y\in \cX}c^\top y)\leq 0.
        \end{align*}
        All of the terms in the sum are positive, and are consequently equal to $0$. Hence we have $x_1,\dots,x_k\in \arg\min_{y\in \cX}c^\top y \subset \dualpoly{\cX}{\cC}$. Consequently, we have $x\in \conv(\cX^\angle \cap \dualpoly{\cX}{\cC})$. Hence, we have 
        \begin{align*}
            \dir{\dualpoly{\cX}{\cC}}&\subset \dir{\conv(\cX^\angle \cap \dualpoly{\cX}{\cC})}\\&=\dir{X^\angle \cap \dualpoly{\cX}{\cC}} \\&\subset \dir{\dualpoly{\cX}{\cC}}.
        \end{align*}
        This proves the desired equality.
    \end{itemize}
\end{proof}

\subsection{Proof of Proposition \ref{prop:equivalence-argmin}}\label{prop:equivalence-argmin proof}
Before proving the proposition, we need to introduce the following lemma.
\begin{lemma}\label{existence of unique argmin}
    For any $x^\star\in \cX^\angle$, there exists $c\in \R^d$ such that $\arg\min_{x\in \cX}c^\top x=\{x^\star\}$, i.e. for all $\delta \in D(x^\star)$, $c^\top \delta >0$.
\end{lemma}
\begin{proof}
    Let $x^\star \in \cX^\angle$. Assume that such a $c\in \R^d$ does not exists. We first show that there exists $\delta^\star \in D(x^\star)$ such that $\Lambda(x^\star)\perp \delta^\star$. Suppose no such $\delta^\star$ exists, then for any $\delta \in D(x^\star)$, there would exist $v(\delta)\in \Lambda(x^\star)$ such that $v(\delta)^\top \delta >0$. Consequently, we have for any $\delta \in D(x^\star)$,
    \begin{align*}
        \left(\sum_{\delta'\in D(x^\star)}^{}v(\delta')\right)^\top \delta >0,
    \end{align*}
    which contradicts our initial assumption. 
    
    Let $N\in \mathbb N$ and $\delta_1,\dots,\delta_N$ such that $D(x^\star)=\{\delta_1,\dots,\delta_N\}$. Assume without loss of generality that $ \Lambda (x^\star)\perp \delta_N$. Consequently, we have for all $c \in \Re^d$
    \begin{align*}
        \left(\forall i \in [N-1],\; c^\top \delta_i\geq 0\right)\implies c^\top \delta_N\leq 0.
    \end{align*}
    We show that this implies that $-\delta_N$ belongs to the cone spanned by $\delta_1,\dots,\delta_{N-1}$, i.e. there exists $\mu_1,\dots,\mu_{N-1}\in \R^+$ such that $\sum_{i=1}^{N-1}\mu_i\delta_i=-\delta_N$. Assume that this is not true. Let $K$ be the cone spanned by $\delta_1,\dots,\delta_{N-1}$. Since $-\delta_N\not\in K$, then (by the separation lemma), there exists $u\in \R^d$ such that for all $h\in K$, we have $u^\top h\geq 0$ and $-u^\top \delta_N <0$. In particular, we have for all $i\in [N-1],$ $u^\top \delta_i\geq 0$ and $u^\top \delta_N>0$, a contradiction. Hence, there exists $\alpha_1,\dots,\alpha_{N-1}\in \R^+$ such that 
    \begin{align*}
        -\delta_N=\sum_{i=1}^{N-1}\alpha_i \delta_i.
    \end{align*}

    Consequently, both $\delta_N$ and $-\delta_N$ are feasible directions from $x^\star$ in $\cX$, which contradicts the fact that $x^\star$ is an extreme point.
\end{proof}

We now prove Proposition \ref{prop:equivalence-argmin}.
\begin{proof}

    It is easy to see that \ref{item:fullargmin} implies \ref{item:singlesol}. We now prove that \ref{item:singlesol} implies \ref{item:fullargmin}. Assume that \ref{item:singlesol} is verified but not \ref{item:fullargmin}, that is $\cD$ is not a \sufficient{}. From Theorem \ref{thm: relatively open characterization}, there exists $\delta\in \Delta(\cX,\cC)$ such that $\delta \not\perp  (\Vect \cD)^\perp$. By definition, there exists $x^\star \in \cX^\angle$ and $c\in \cC$ such that $c\in F(x^\star,\delta)$. From Lemma \ref{existence of unique argmin}, there exists $v\in \R^d$ such that for all $\delta \in D (x^\star)$, $v^\top \delta >0$. Let $\varepsilon>0$ such that $B(c,\varepsilon)\subset \cC$. Let $\eta>0$ small enough such that $c+\eta v\in B(c,\varepsilon)$, and $\eta '$ small enough such that $c_{\eta,\eta'}:=c+\eta v - \eta '\delta_{(\Vect \cD)^\perp}\in B(c,\varepsilon)$. For any $\delta' \in D(x^\star)$, we have
    \begin{align*}
        (c+\eta v)^\top \delta' = \underbrace{c^\top \delta'}_{\geq 0} +\underbrace{\eta v^\top \delta'}_{>0} >0.
    \end{align*}
    This means that $\arg\min_{x\in \cX}(c+\eta v)^\top x=\{x^\star\}$. Furthermore, we have 
    \begin{align*}
        c_{\eta,\eta'}^\top \delta = \underbrace{c^\top \delta}_{=0, \text{ as }c\in F(x^\star,\delta)} +\eta v^\top \delta - \eta ' \underbrace{\norm{\delta_{ (\Vect(\cD))^\perp}\textbf{}}^2}_{\neq 0, \text{ as }\delta \not\perp  (\Vect \cD)^\perp }.
    \end{align*}
    
    Consequently, when $\eta$ is small enough compared to $\eta'$, we have $c_{\eta,\eta'}^\top \delta<0$, i.e. $c_{\eta,\eta'} \not\in \Lambda(x^\star)$. This means that $x^\star \not\in \arg\min_{x\in \cX}c_{\eta,\eta'}^\top x $. Assume that a mapping $\hat{x}$ satisfying condition \ref{item:singlesol} of the proposition. We have $c+\eta v-c_{\eta,\eta'}=\eta'\delta_{ (\Vect(\cD))^\perp}\in (\Vect(\cD))^\perp$. This means that for all $i\in [N]$, we have $(c+\eta v)^\top q_i=c_{\eta,\eta'}^\top q_i$, and hence 
    \begin{align*}
        \hat{x}((c+\eta v)^\top q_1,\dots,(c+\eta v)^\top q_N)=\hat{x}(c_{\eta,\eta'}^\top q_1,\dots,c_{\eta,\eta'}^\top q_N),
    \end{align*}
    which implies that
    \begin{align*}
        \hat{x}(c_{\eta,\eta'}^\top q_1,\dots,c_{\eta,\eta'}^\top q_N)\in \left(\arg\min_{x\in \cX}c_{\eta,\eta'}^\top x\right)\cap \left(\arg\min_{\cX}(c+\eta v)^\top x\right)=\varnothing,
    \end{align*}
    which is impossible.

\end{proof}

% \subsection{Proof of Proposition \ref{data collection hardness}}
% \begin{proof}
%     We will prove that an instance ofthe problem can be reduced to an NP-Hard problem. Assume that $\cQ$ is a finite set and that $B=\{0\}$, and that there exists $v\in \R^d$ such that $A=\Vect \{v\}$. This instance is equivalent to the following problem: find the smallest set $\cD\subset \cQ$ such that $v\in \Vect \cD$, i.e. to the problem
%     \begin{align*}
%         \min_{\phi\in \R^{\abs{\cQ}}} \norm{\phi}_0 \text{ s.t. }M\phi=v,
%     \end{align*}
%     where $M$ is a matrix whose columns are the elements of $\cQ$. This problem is well known to be NP-Hard \citep{natarajan1995sparse}.
% \end{proof}

\section{Proof of \cref{thm:alg_termination}: Correctness}
\begin{proof}\;
\begin{itemize}
    \item We first show that when the algorithm terminates, i.e., the condition of the while loop is no longer satisfied, then with probability 1, $\dir{\dualpoly{\cX}{\cC}}\subset\Vect \cD$.
    Notice that the constraints in the minimization and maximization problems in Algorithm  \ref{algorithm to compute delta} encode complimentary slackness and, therefore are equivalent to $$\min / \max \{ \alpha^\top \proj{(\Vect \cD)^\perp}{x^\star - x_0}  \; : \; c \in \cC, \; x^\star \in \argmin_{x \in \cX} c^\top x\}.$$
    By definition of $\dualpoly{\cX}{\cC}$, this equivalent to $$\min / \max \{ \alpha^\top \proj{(\Vect \cD)^\perp}{x^\star - x_0}  \; : \; x^\star \in \dualpoly{\cX}{\cC}\}.$$
    
     If the two problems have an optimal value equal to $0$, then $\proj{(\Vect \cD)^\perp}{\dir{\dualpoly{\cX}{\cC}}}\perp \alpha$ i.e. $\alpha \in \proj{(\Vect \cD)^\perp}{\dir{\dualpoly{\cX}{\cC}}}^\perp$. Unless $\proj{(\Vect \cD)^\perp}{\dir{\dualpoly{\cX}{\cC}}}^\perp=\R^d$, this set is of empty interior and its Lebesgue measure is equal to $0$, and consequently the probability of having $\proj{(\Vect \cD)^\perp}{\dir{\dualpoly{\cX}{\cC}}}\perp \alpha$ is zero since $\alpha$ has a continuous distribution. Hence, with probability $1$, we have $\proj{(\Vect \cD)^\perp}{\dir{\dualpoly{\cX}{\cC}}}=\{0\}$ i.e. $\dir{\dualpoly{\cX}{\cC}}\subset\Vect \cD$.

    \item  We now show that at every step of the algorithm, the dimension of the span of $\cD$ increases by $1$, and that it remains a linearly independent set, as well as satisfies $\Vect \cD \subset \dir{\dualpoly{\cX}{\cC}}$. Indeed, initially, $\cD$ is empty and is hence a linearly independent set and satisfies $\Vect \cD \subset \dir{\dualpoly{\cX}{\cC}}$. Assuming that $\cD$ is a linearly independent set and that $\Vect \cD \subset \dir{\dualpoly{\cX}{\cC}}$, if there exists $x\in \dualpoly{\cX}{\cC}$ such that $\alpha^\top \proj{(\Vect \cD)^\perp}{x_0-x}\neq 0$, then $\proj{(\Vect \cD)^\perp}{x_0-x}\neq 0$ with probability $1$ and consequently $x_0-x\in \dir{\dualpoly{\cX}{\cC}}\setminus \Vect \cD$. Hence, we have $\dim (\Vect (\cD\cup\{x_0-x\}))=\dim (\Vect \cD)+1$ and $\cD\cup\{x_0-x\}$ is a linearly independent set and satisfies $\Vect (\cD\cup\{x_0-x\}) \subset \dir{\dualpoly{\cX}{\cC}}$, which proves the desired result.
    \item Finally, combining the two results above, when the algorithm terminates, $\cD$ is a linearly independent set, and $\Vect \cD=\dir{\dualpoly{\cX}{\cC}}$ i.e. $\cD$ is a basis of $\dir{\dualpoly{\cX}{\cC}}$ with probability 1. Furthermore, the analysis above show that the algorithm indeed terminates after $\dim \dir{\dualpoly{\cX}{\cC}}$ iterations of the while loop.
\end{itemize}

\end{proof}

\section{Useful Lemmas}

  \begin{lemma} \label{cone equivalence}
            Assume that $\cX$ is bounded. For every $c\in \R^d$, $\arg\min_{x\in \cX}c^\top x$ is a polyhedron, and all of its extreme points are extreme points in $\cX$. Recall the optimality cones $\Lambda(x^\star)$ of all $x^\star \in \cX^\angle$ defined in \cref{optimality cone def}.  For every $c,c'\in \R^d$, the following equivalence holds.
            \begin{align*}
                \arg\min_{x\in \cX}c^\top x=\arg\min_{x\in \cX}c'^\top x \Longleftrightarrow \forall x \in \cX^\angle,\; \left(c\in \Lambda (x)\Longleftrightarrow c'\in \Lambda(x)\right).
            \end{align*}
        \end{lemma}
        
        \begin{proof}
    
            We first show that for any $c\in \R^d$, any extreme point in $\arg\min_{x\in \cX}c^\top x$ is in $\cX^\angle$. Let $x\in \cX$ be an extreme point of $\arg\min_{x\in \cX}c^\top x$. Assume that $x$ is not an extreme point in $\cX$. Hence, there exists $u\in \R^d$ such that $x\pm u\in \cX$. If $c^\top u\neq 0$ we have $c^\top (x-u)<c^\top x$ or $c^\top (x+u)<c^\top x$. This means that $x\not\in \arg\min_{x\in \cX}c^\top x$ which is impossible.
            If $c^\top u=0$, then $c\pm u\in \arg\min_{x\in \cX}c^\top x$, which is also impossible since $x$ is an extreme point in $\arg\min_{x\in \cX}c^\top x$. Hence, since $\arg\min_{x\in \cX}c^\top x$ is convex, it's the convex hull of its extreme points.

            Consequently, for any $c,c'\in \R^d$, $\arg\min_{x\in \cX}c^\top x=\arg\min_{x\in \cX}c'^\top x$ if and only if these two sets have the same set of extreme points. Furthermore, for any $x\in \cX^\angle$, we have $c\in \Lambda(x)$ if on and only if $x\in \arg\min_{x\in \cX}c^\top x$. Hence, the desired result immediately follows:
            \begin{align*}
                \arg\min_{x\in \cX}c^\top x=\arg\min_{x\in \cX}c'^\top x&\Longleftrightarrow \left(\forall x\in \cX^\angle,\; x\in \arg\min_{x\in \cX}c^\top x \Longleftrightarrow x\in \arg\min_{x\in \cX}c'^\top x \right)\\
                & \Longleftrightarrow \left(\forall x \in \cX^\angle,\; c\in \Lambda(x)\Longleftrightarrow c'\in \Lambda(x)\right)
            \end{align*}
        \end{proof}
\begin{lemma} \label{lemma:feasible directions in general polyhedron}
            Let $r=\dim F_0 \cap \Ker A$. For any $x\in \cX$, there exists a set $V=\{\delta_1,\dots,\delta_r\}\subset \FD(x)$, such that $V$ is a basis of $F_0 \cap \Ker A$. In particular, the set of extreme directions of $\FD(x)$ spans $F_0\cap \Ker A$.
        \end{lemma}
\begin{proof}
            Let $x\in \cX$. Let $\cX^\angle$ be the set of extreme points of $\cX$, and $D^\angle$ be the set of extreme rays of $\cX$. For every $x^\angle \in \cX^\angle$ and $\delta^\angle \in D^\angle$, we have $x^\angle \geq 0$ and $\delta^\angle \geq 0$. Let $\{\alpha_{x^\angle}\}_{x^\angle \in \cX^\angle}\subset \R^*_+$ and $\{\alpha_{\delta^\angle}\}_{\delta^\angle \in D^\angle}\subset \R^*_+$ be a set of strictly positive numbers such that $\sum_{x^\angle \in \cX^\angle}^{}\alpha_{x^\angle}=1$. We define 
            \begin{align*}
                \overline{x}=\sum_{x^\angle \in \cX^\angle}^{}\alpha_{x^\angle}x^\angle + \sum_{\delta^\angle \in D^\angle}^{}\alpha_{\delta^\angle}\delta^\angle\in \cX.
            \end{align*}

        We have for any $i\in I_0$, $\overline{x}_i>0$. Indeed, for any $i\in I_0$, if $\overline{x}_i=0$, then for every $x^\angle \in \cX^\angle$ and $\delta^\angle \in D^\angle$, $x_i^\angle=\delta_i^\angle =0$ and hence for any $x'\in \cX$, $x'_i=0$ i.e. $i\not \in I_0$ which is impossible. Let $$\varepsilon:=\frac{1}{2}\min_{i\in I_0}\overline{x}_i>0,$$ for any $\delta \in F_0 \cap \Ker A$ such that $\norm{\delta}<\varepsilon$, we have $A(\overline{x}+\delta)=A\overline{x}=b$, and for every $i\in [d]$, if $i\in I_0$, then $\overline{x}_i + \delta _i >0$ and if $i\not\in I_0$, $\overline{x}_i=\delta_i=0$ and consequently $x+\delta \geq 0$, i.e. $\delta$ is a feasible direction for $\overline{x}$. Hence, every element of $B(0,\varepsilon)\cap F_0 \cap \Ker A$ is a feasible direction for $\overline{x}$, and consequently any element of $F_0 \cap \Ker A$. Let $x\in \cX$, and $v_1,\dots, v_r$ a basis of $F_0\cap \Ker A$ such that for every $i\in [r],\; \norm{v_i}=1$. Let $\eta \in \R^*_+$ small enough such that $\forall i \in [r]$, $\overline{x} + \eta \delta _i \in \cX$. We would like to show that for a well-chosen value of $\eta$, the following set of feasible directions for $x$, $\{\overline{x}+\eta v_1 - x,\dots,\overline{x}+\eta v_r - x,\}$ is a basis of $F_0 \cap \Ker A$. Since $\overline{x}-x\in F_0 \cap \Ker A$, we consider $\beta_1,\dots,\beta_r\in \R$ such that 
        \begin{align*}
            \overline{x}-x=\sum_{i=1}^{r}\beta_i v_i.
        \end{align*}
        Let $\alpha_1,\dots,\alpha_r\in \R$. We have 
        \begin{align*}
            \sum_{i=1}^{r}\alpha_i(\overline{x}+\eta v_i-x)=0 &\Longrightarrow \underbrace{\left(\sum_{i=1}^{r}\alpha_i\right)}_{:=A}(\overline{x}-x)+\eta \sum_{i=1}^{r}\alpha_i v_i=0\\
            &\Longrightarrow \sum_{i=1}^{r}(A\beta_i + \eta \alpha_i)v_i=0\\
            &\Longrightarrow \forall i \in [r],\; A\beta_i + \eta \alpha_i=0\\
            &\Longrightarrow \underbrace{ \forall i\in [r], \alpha_i=0 \text{ or }\left( A\neq 0 \text{ and } \forall i \in [r],\;\frac{\alpha_i}{A}=-\frac{\beta_i}{\eta}\right)}_{(*)}.
        \end{align*}
        By summing over $i$ in the last equality above, we get
        \begin{align*}
            (*)&\Longrightarrow  \forall i\in [r], \alpha_i=0 \text{ or }\left( A\neq 0 \text{ and } 1=-\frac{1}{\eta}\smash{\underbrace{\sum_{i=1}^{r}\beta_i}_{B}} \right)
        \end{align*}
        
        \vspace{5mm}
        
        The equality above is equivalent to $\eta = -B$. Hence, it suffices to take $\eta\neq -B$ and $\eta$ small enough to ensure that $\{\overline{x}+\eta v_1 - x,\dots,\overline{x}+\eta v_r - x,\}$ is linearly independent and is a set of feasible directions for $x$, and consequently since all of these vectors are elements of $F_0 \cap \Ker A$, and there are $r=\dim F_0 \cap \Ker A$ of them, $\{\overline{x}+\eta v_1 - x,\dots,\overline{x}+\eta v_r - x,\}$ is indeed a set of feasible directions for $x$ that is a basis of $F_0\cap \Ker A$.
        \end{proof}
 
%%%%%%%%%%%%%%%%%%%%%%%%%%%%%%%%%%%%%%%%%%%%%%%%%%%%%%%%%%%%
\section{Further Notes on Experiments of \cref{sec:experiments}} \label{appendix: additional notes on experiments}

\textbf{Note.}
The hiring scenarios considered in this paper are stylized decision models intended to illustrate how data informativeness depends on task structure and prior uncertainty. While some formulations include group-based constraints (e.g., per-category quotas), these are not meant to prescribe or endorse any specific hiring policy.

\end{document}